%% file: arxiv-nonlinear_VI-alnashri-droniou.tex
\documentclass[final,reqno]{amsart}

\usepackage{graphicx}
\usepackage{amsmath,amsthm,amsfonts,amscd,amssymb}
\usepackage[color,notref,notcite]{showkeys}
\usepackage{url}
\usepackage{subeqnarray}
\usepackage[color,notref,notcite]{showkeys} 
\definecolor{labelkey}{rgb}{0.6,0,1}
\usepackage{enumitem}
\usepackage{algorithm}
\usepackage{algpseudocode}
\usepackage{citesort}

\theoremstyle{plain}
\newtheorem{theorem}{Theorem}[section]

\newtheorem{lemma}[theorem]{Lemma}

\newtheorem{proposition}[theorem]{Proposition}

\newtheorem{assumptions}[theorem]{Assumptions}

\theoremstyle{definition}
\newtheorem{definition}[theorem]{Definition}

\def\bhyp#1{\begin{equation}\label{#1}\begin{array}{c}}
\def\ehyp{\end{array}\end{equation}}

\newcounter{cst}

\theoremstyle{remark}
\newtheorem{remark}[theorem]{Remark}

\numberwithin{equation}{section}
\numberwithin{figure}{section}

\newcommand{\RR}{{\mathbb R}}
\newcommand{\NN}{{\mathbb N}}

\def\O{\Omega}
\def\dsp{\displaystyle}
\def\bfn{\mathbf{n}}
\def\bfa{\mathbf{a}}
\def\disc{{\mathcal D}}
\def\mesh{{\mathcal M}}
\newcommand{\polyd}{{\mathcal T}}
\def\edges{{\mathcal E}}
\def\edge{\sigma}
\def\xcv{x_K}
\def\cv{K}

\newcommand{\edgescv}{{{\edges}_K}}  
\newcommand{\edgesext}{{{\edges}_{\rm ext}}} 
\newcommand{\edgesint}{{{\edges}_{\rm int}}} 
\newcommand{\centers}{\mathcal{P}}

\def\dr{\partial}
\newcommand{\centeredge}{\overline{x}_\edge} 
\newcommand{\cH}{{\mathcal H}}
\newcommand{\cK}{{\mathcal K}}
\newcommand{\cI}{{\mathcal I}}
\DeclareMathOperator*{\argminB}{argmin}

\def\bvarphi{\boldsymbol{\phi}}
\def\bG{\mathbf{G}}

\newif\ifcorr\corrtrue
\usepackage[normalem]{ulem}
\definecolor{violet}{rgb}{0.580,0.,0.827}

\def\bpsi{{\boldsymbol \psi}}

\newcommand{\ud}{\, \mathrm{d}} 
\def\div{\mathop{\rm div}}

\title[Analysis of schemes for non-linear variational inequalities]{A gradient discretisation method to analyse numerical schemes for non-linear variational inequalities, application to the seepage problem}

\author{Yahya Alnashri}
\address[Yahya Alnashri]{School of Mathematical Sciences, Monash University, Victoria 3800, Australia\\
and Umm-Alqura University}
\email{yanashri@ummalqura.edu}

\author{J\'er\^ome Droniou}
\address[J\'er\^ome Droniou]{School of Mathematical Sciences, Monash University, Victoria 3800, Australia.}
\email{jerome.droniou@monash.edu}



\date{\today}

\begin{document}
\newcommand{\subscript}[2]{$#1 _ #2$}

\begin{abstract}
Using the gradient discretisation method (GDM), we provide a complete and unified numerical analysis for non-linear variational inequalities (VIs) based on Leray--Lions operators and subject to non-homogeneous Dirichlet and Signorini boundary conditions. This analysis is proved to be easily extended to the obstacle and Bulkley models, which can be formulated as non-linear VIs. It also enables us to establish convergence results for many conforming and nonconforming numerical schemes included in the GDM, and not previously studied for these models. Our theoretical results are applied to the hybrid mimetic mixed method (HMM), a family of schemes that fit into the GDM. Numerical results are provided for HMM on the seepage model, and demonstrate that, even on distorted meshes, this method provides accurate results.
\end{abstract}

\maketitle


\section{Introduction}\label{sec:intro}
\par Non-linear variational inequalities are related to a wide range of applications. In particular, unconfined seepage models, free boundary problems involving Signorini boundary conditions,  can be used to study the construction of earth dams, embankments and hydraulic design. With non-linear variational inequalities, one can also study the Bulkley fluid model, which is applicable to different phenomena and processes, such as blood flow \cite{N-7-app}, food processing \cite{A3} and Bingham fluid flows \cite{N-21-Louis}. 

We consider here variational inequalities (VIs) related to elliptic equations of the type
\begin{subequations}
\label{model:ll}
\begin{align}
-\div\bfa(x,u,\nabla u)&=f \quad \mbox{in } \O,\label{eq:pde}\\ 
u&=g \quad \mbox{on } \partial\O,
\end{align}
\end{subequations}
where $\Omega$ is an open bounded connected subset of $\mathbb{R}^{d}$, $d\ge 1$, with boundary $\partial\O$. Precise assumptions on data will be stated in the next sections. The purpose of this paper is to provide a complete and unified convergence analysis of numerical schemes for VIs based on \eqref{model:ll}. Our convergence result applies to a wide range of methods, such as finite elements methods (conforming and non-conforming), finite volume methods, mimetic finite difference schemes, etc. To our knowledge, this result is the first one for non-conforming methods applied
to non-linear variational inequalities.

The theory on PDEs of the kind \eqref{model:ll} has been covered in several works, see \cite{N-15,N-16,N-17,N-20} and references therein. A number of numerical analyses
on these models has also been carried out, starting from the approximation
of the $p$-Laplace equation, with proven rates of convergences, by $\mathbb{P}_1$ finite
elements in \cite{N-11}. Subsequent works consider more general Leray--Lions models,
possibly transient, and establish either error estimates (under regularity assumptions
on the solution to the PDE), or prove the convergence towards a solution with
minimal regularity. We refer the reader to 
\cite{B1,N-16,DiPietro-Droniou:16,DiPietro-Droniou:15,Glowinski.Rappaz:03,Liu.Yan:01,Andreianov.Boyer.ea:05,Andreianov.Boyer.ea:07} for a few examples.
Several algorithms can be used to compute the solution to the corresponding non-linear
numerical schemes, from basic fixed-point iterations (which corresponds to the
Ka\c{c}anov method \cite{N-20}) to Newton methods, to multigrid techniques \cite{N-10}, to augmented Lagrangian algorithms \cite{N-12}.

The mathematical theory of variational inequalities based on equations of the kind \eqref{model:ll} is well understood, see e.g. \cite{NB-1,N-3,N-4,N-18}. 
We note that \cite{N-4} considers an obstacle problem with measure source terms rather than
$W^{1,p}(\O)'$ source terms (the theory for the corresponding PDEs is developed in
\cite{N-14}). \cite{N-9} studies non-linear quasi-variational inequalities and proposes a semi-smooth Newton iteration to obtain a solution. 

The numerical approximation of variational inequalities based on linear operators, including the issues faced
with numerical approximations of the convex set described by the obstacle, has been covered in a number
of works -- see, e.g., \cite{G1} and references therein. Some works tackle the question of
the numerical approximation of VIs based on non-linear equations such as \eqref{model:ll}. Under strong monotonicity assumptions on the operator, \cite{A4} develops a convergence analysis of conforming numerical schemes for non-linear VIs. \cite{A-2-19} develops the analysis of conforming finite elements method for VIs involving a non-linear proper function. In \cite{N-8,nl-VI-1994}, $\mathbb P_1$ finite elements are applied to the obstacle problem for a $p$-Laplace-like
operator, with homogeneous Dirichlet boundary conditions and zero barrier inside the domain; an \emph{a priori} error estimate is obtained under $W^{2,p}$ regularity on the solution. See also
\cite{nl-VI-1995} for non-linear parabolic variational inequalities. The Bulkley model also has been approximated by $\mathbb P_1$ finite elements \cite{N-19-app,N-22} and Lagrange methods \cite{N-19-app}.
In \cite{A1}, a seepage model is approximated by a finite elements method,
but no convergence analysis is carried out. The authors utilise a fixed point method (Ka\c{c}anov) to treat the non-linearity and compute the solution to the scheme.

All these studies of numerical schemes for non-linear VIs deal with conforming numerical schemes, mostly $\mathbb{P}_1$ finite elements. It seems that a lot of work remains to be done, starting from convergences analyses and tests on other
kinds of schemes than conforming finite element (FE) schemes (e.g., non-conforming FE, finite volume, mimetic finite differences, etc.).
Our work aims at filling this gap. We provide a complete convergence analysis of numerical schemes
for variational inequalities based on non-linear Leray--Lions operators, and we
present numerical results using the hybrid mimetic mixed method. This method, contrary to
FE methods, is applicable on grids with very general cell geometries as encountered in some
porous flow applications.

Instead of conducting individual studies for each numerical scheme, we develop a
unified convergence analysis that is readily applicable to several methods.
This is done by adapting the gradient discretisation method (GDM) to non-linear VIs.
The GDM is a framework for the analysis of numerical schemes for diffusion PDEs.
It covers a variety of methods, such as conforming, non-conforming and mixed finite elements methods (including the non-conforming ``Crouzeix--Raviart'' method and the  Raviart--Thomas method), hybrid mimetic mixed methods (which contain hybrid mimetic finite differences,
hybrid finite volumes/SUSHI scheme and mixed finite volumes), nodal mimetic finite differences,
and finite volumes methods (such as some multi-points flux approximation and discrete duality finite volume methods). The original GDM identifies a small number of properties required to
establish the convergence of numerical schemes for various models based
on elliptic and parabolic PDEs: linear and non-linear diffusion, stationnary and transient
Leray--Lions equations, the Stefan model of melting material, the Richards model
of water flow in an unsaturated porous medium, diphasic flows, etc.
The GDM is also adapted to various boundary conditions.
For more details, we refer the reader to the monograph \cite{S1} and to the
papers \cite{B1,B2,B10,zamm2013,eym-12-stef,Droniou.Eymard:15}.

\par In this work, we adapt the gradient discretisation method to three non-linear variational inequalities involving Leray--Lions operators. We show that the GDM provides
a unified convergence analysis of numerical methods for these models. 
This analysis yields convergence theorems of numerical schemes for
meaningful models of VIs, including the non-linear seepage problems and the Bulkley model.
To illustrate our theoretical results, we apply the hybrid mimetic mixed (HMM) method
to the seepage problem, and show that -- even on distorted meshes -- its efficiency is comparable
to the $\mathbb{P}_1$ finite elements of \cite{A1}. One of its additional strengths, however,
is that it is applicable on very generic meshes, contrary to the $\mathbb{P}_1$ finite element
method. As proved in \cite{B5}, the HMM method contains the hybrid finite volume method of \cite{sushi}, the (mixed/hybrid) mimetic finite difference method of \cite{bre-05-fam} and the mixed finite volume method of \cite{dro-06-mix}.

This paper is organised as follows. Section \ref{sec:nlsign} details the non-linear Signorini problem, its approximation by the gradient discretisation method,
and the corresponding convergence results. Section \ref{sec:obs} shows that the GDM can successfully be adapted to the obstacle problem and the Bulkley fluid model. 
A short section, Section \ref{sec:approx.barriers}, describes the case where the
barriers of the Signorini and obstacle problems are approximated as part of the
discretisation process.
In Section \ref{sec:HMM} we show that our result apply to the HMM scheme, and establish
its convergence for all three models. Section \ref{sec:test} present numerical tests that
demonstrate the efficiency of the HMM method for solving the seepage model on various meshes,
including very distorted ones.
An appendix, Section \ref{sec:appen}, presents an interpolation operator useful for the HMM method
(and, more generaly, for methods based on cell and face unknowns).


\section{Non-linear Signorini problem}\label{sec:nlsign}

\subsection{Continous problem}
\par We first consider the following non-linear Signorini problem:
\begin{align}
-\div\bfa(x,\bar u,\nabla\bar u)= f &\mbox{\quad in $\Omega$,} \label{nlsign1}\\
\bar{u} =g &\mbox{\quad on $\Gamma_{1}$,} \label{nlsign2}\\
\bfa(x,\bar u,\nabla\bar u)\cdot \mathbf{n}= 0 &\mbox{\quad on $\Gamma_{2}$,} \label{nlsign3}\\
\left.
\begin{array}{r}
\dsp \bar{u} \leq a\\
\dsp \bfa(x,\bar u,\nabla\bar u)\cdot \mathbf{n} \leq 0\\
\dsp (a-\bar{u})\bfa(x,\bar u,\nabla\bar u)\cdot \mathbf{n}  = 0
\end{array} \right\} 
& \mbox{\quad on}\; \Gamma_{3}.
\label{nlcond}
\end{align}
Here $\mathbf{n}$ denotes the unit outer normal to the boundary $\partial\O$, which is
split in three parts ($\Gamma_1, \Gamma_2, \Gamma_3)$.
The assumptions on the Leray--Lions operator $\bfa$ are standard:
\begin{equation}\label{asm:carth}
\mbox{$\bfa : \O \times \RR \times \RR^d \rightarrow \RR^d$ is a Caratheodory function},
\end{equation}
(i.e., for a.e.\ $x \in \O$, $(s,\xi)\mapsto \bfa(x,s,\xi)$ is continuous
and, for all $(s,\xi) \in \RR\times \RR^d$, $x \rightarrow \bfa(x,s,\xi)$ is measurable) and, for some $p\in(1,\infty)$ and $p'=\frac{p}{p-1}$, 
\begin{equation}\label{asm:grow}
\begin{aligned}
&\exists \overline{a} \in L^{p'}(\O), \exists \mu >0 :\\
&|\bfa(x,s,\xi)|\leq \overline{a}(x)+\mu |\xi|^{p-1}, \mbox{ for a.e. }x \in \O, \; \forall s \in \RR,\; \forall \xi \in \RR^d,
\end{aligned}
\end{equation}
\begin{equation}\label{asm:coer}
\exists \underline{a}>0: \bfa(x,s,\xi)\cdot\xi \geq \underline{a}|\xi|^p, \mbox{ for a.e. }x \in \O, \; \forall s \in \RR,\; \forall \xi \in \RR^d,
\end{equation}
\begin{equation}\label{asm:mont}
(\bfa(x,s,\xi)-\bfa(x,s,\chi))\cdot (\xi-\chi) \geq 0 \mbox{ for a.e. }x \in \O, \; \forall s \in \RR,\; \forall \xi, \chi \in \RR^d.
\end{equation}
Assumptions \eqref{asm:grow}, \eqref{asm:coer} and \eqref{asm:mont} are respectively called the growth, coercivity and monotonicity conditions. Setting $\bfa(x,u,\nabla u)= |\nabla u|^{p-2}\nabla u$ 
in \eqref{nlsign1} gives in particular the $p$-Laplacian operator. 

\begin{remark}\label{rem:HK}
With $p=2$ and $\bfa(x,u,\nabla u)=\mathcal H^{\lambda}_\varepsilon(u-h(x))\mathbf{K}(x)\nabla u$, with $\mathcal H^{\lambda}_\varepsilon$ the regularised Heavyside function defined in \eqref{def:Hreg}, $h$ a fixed function and $\mathbf{K}$ the permeability tensor, Problem \eqref{nlsign1}--\eqref{nlcond} covers seepage models. The role of the regularised Heaviside function is
to extend the Darcy law to the dry domain. We refer the reader to \cite{A1} and references therein for more details.
\end{remark}


\begin{assumptions}\label{hip:nosign} The assumptions on the data in Problem \eqref{nlsign1}--\eqref{nlcond} are the following:
\begin{enumerate}
\item the operator $\bfa$ satisfies \eqref{asm:carth}--\eqref{asm:mont} and the domain $\O$ has a Lipschitz boundary,
\item the parts of the boundary, $\Gamma_{1}, \Gamma_{2}$ and $\Gamma_{3}$, are assumed to be measurable and pairwise disjoint
subsets of $\partial\Omega$ such that $\Gamma_{1}\cup \Gamma_{2} \cup \Gamma_{3} = \partial\Omega$ and the $(d-1)$-dimensional measure of $\Gamma_{1}$ is non zero,
\item the source term $f$ belongs to $L^{p'}(\Omega)$, the barrier $a$ belongs to $L^p(\partial\Omega)$ and the boundary data $g$ belongs to $W^{1-\frac{1}{p},p}(\partial\O)$,
\item the closed convex set $\mathcal{K}:=\{ v \in W^{1,p}(\Omega)\; : \; \gamma(v)=g\; \mbox{on}\; \Gamma_{1},\; \gamma(v) \leq a\; \mbox{on}\; \Gamma_{3}\}
$ is non-empty.
\end{enumerate}
\end{assumptions}

Based on Assumption \ref{hip:nosign}, Problem \eqref{nlsign1}--\eqref{nlcond} can be written in the following weak sense:
\begin{equation}\label{wenlsign}
\left\{
\begin{array}{ll}
\dsp \mbox{Find}\; \bar{u} \in \mathcal{K} \mbox{ such that, } \forall v\in\cK,\\
\dsp\int_\O \bfa(x,\bar u,\nabla\bar u) \cdot \nabla(\bar{u}-v)\ud x
\leq \int_\O f(\bar{u}-v)\ud x.
\end{array}
\right.
\end{equation}
The existence of a solution to Problem \eqref{wenlsign} is ensured
by \cite[Theorem 8.2, Chap. 2]{NB-1}.


\subsection{The gradient discretisation method}
The GDM consists in replacing, in the weak formulation of the model, the continuous
space and operators by discrete ones, obtaining thus a gradient scheme (GS).
The discrete elements are gathered in what is called a gradient discretisation (GD).
The restrictions put on these discrete elements are rather light, and there
are therefore a large choice of possible GDs. It was shown in previous papers
(see \cite{B10} for a review) that, for a number of classical schemes,
specific GDs can be chosen such that the corresponding GSs are the considered
schemes.

\begin{definition}(Gradient discretisation for Signorini BCs)\label{def:gd-nlsign}. A gradient discretisation $\mathcal{D}$ for Signorini boundary conditions and nonhomogeneous Dirichlet boundary conditions is $\mathcal{D}=(X_\disc, \Pi_{\mathcal{D}}, \cI_{\disc,\Gamma_1},\mathbb{T}_{\mathcal{D}}, \nabla_{\mathcal{D}})$, where:
\begin{enumerate}
\item the set of discrete unknowns $X_\disc= X_{\disc,\Gamma_{2,3}}\oplus X_{\disc,\Gamma_1}$ is a direct sum of two finite dimensional spaces on $\RR$. The first space corresponds to the interior degrees of freedom and to the boundaries degrees of freedom on $\Gamma_2 \cup \Gamma_3$. The second space corresponds to the boundary degrees of freedom on $\Gamma_1$, 
\item the linear mapping $\Pi_{\mathcal{D}} : X_{\mathcal{D}} \rightarrow L^{p}(\varOmega)$ 
reconstructs functions from the degrees of freedom,
\item the linear mapping $\cI_{\disc,\Gamma_1} : W^{1-\frac{1}{p},p}(\partial\O) \rightarrow X_{\disc,\Gamma_1}$ interpolates the traces of functions in $W^{1,p}(\O)$ on the
degrees of freedom,
\item the linear mapping $\mathbb{T}_{\mathcal{D}} : X_{\mathcal{D}}\rightarrow L^p(\partial\O) $ reconstructs traces from the degrees of freedom,
\item the linear mapping $\nabla_{\mathcal{D}} : X_{\mathcal{D}} \rightarrow L^{p}(\Omega)^{d}$ 
reconstructs gradients from the degrees of freedom.
It must be such that $\|\nabla_{\mathcal{D}}\cdot\|_{L^{p}(\Omega)^{d}}$ is a norm on $X_{\disc,\Gamma_{2,3}}$.
\end{enumerate}
\label{def:gdnhseepage}
\end{definition}

As already explained, the GS is obtained by taking the weak formulation
\eqref{wenlsign} of the model, and replacing the continuous elements (space, function, gradient, trace...) by the discrete elements provided by the chosen GD.

\begin{definition}(Gradient scheme for Signorini problem). Let $\mathcal{D}$ be a gradient discretisation in the sense of Definition \ref{def:gdnhseepage}. The corresponding gradient scheme for Problem \eqref{wenlsign} is
\begin{equation}\label{gsnlsign}
\left\{
\begin{array}{ll}
\dsp \mbox{Find}\; u \in \mathcal{K}_{\mathcal{D}}\mbox{ such that } \forall v\in\cK_\disc,\\
\dsp\int_\O \bfa(x,\Pi_\disc u,\nabla_\disc u) \cdot\nabla_{\mathcal{D}}(u-v)\ud x
\leq \dsp\int_\O f\Pi_{\mathcal{D}}{(u-v)}\ud x,
\end{array}
\right.
\end{equation}
where $\mathcal{K}_{\mathcal{D}}:=\{ v \in \cI_{\disc,\Gamma_1}g+X_{\disc,\Gamma_{2,3}}\; : \; \mathbb{T}_{\mathcal{D}}v \leq a\; \mbox{on}\; \Gamma_{3}\}$.
\end{definition}

We presented in \cite{AD14} three properties called coercivity, GD-consistency and limit-conformity to assess the accuracy of gradient schemes for VIs. These properties were sufficient to establish error estimates and prove the
convergence of the GDM for VIs based on \emph{linear} differential operator. For non-linear problems, an additional property called compactness is required to ensure the convergence of the GDM. Let us describe these four properties in the context of Signorini boundary conditions.


\begin{definition}[Coercivity]\label{def:nlsigncoer}
If $\disc$ is a gradient discretisation in the sense of Definition \ref{def:gd-nlsign},
set
\begin{eqnarray}
C_{\mathcal{D}} =  {\displaystyle \max_{v \in X_{\mathcal{D},\Gamma_{2,3}}\setminus\{0\}}\Big(\frac{\|\Pi_{\mathcal{D}}v\|_{L^{p}(\Omega)}}{\|\nabla_{\mathcal{D}} v\|_{L^{p}(\Omega)^{d}}}}+ \frac{\|\mathbb{T}_{\mathcal{D}}v\|_{L^p(\partial\O)}}{\|\nabla_{\mathcal{D}} v\|_{L^{p}(\Omega)^{d}}}\Big).
\label{coercivityseepage}
\end{eqnarray}
A sequence $(\mathcal{D}_{m})_{m \in \mathbb{N}}$ of gradient discretisations is \emph{coercive} if $(C_{\mathcal D_m})_{m\in\NN}$ remains bounded.
\end{definition}
\begin{definition}[GD-Consistency]\label{def:nlsigncons}
If $\disc$ is a gradient discretisation in the sense of Definition \ref{def:gd-nlsign}, define
$S_{\mathcal{D}} : \mathcal{K}\to [0, +\infty)$ by
\begin{equation}
\forall \varphi\in \mathcal{K}, \; S_{\mathcal{D}}(\varphi)= 
\min_{v\in \mathcal K_{\disc}}\left(\| \Pi_{\mathcal{D}} v - \varphi \|_{L^{p}(\Omega)} 
+ \| \nabla_{\mathcal{D}} v - \nabla \varphi \|_{L^{p}(\Omega)^{d}}\right).
\label{consistencysig}
\end{equation}
A sequence $(\mathcal{D}_{m})_{m \in \mathbb{N}}$ of gradient discretisations is \emph{GD-consistent} (or simply \emph{consistent}, for short) if $\lim_{m\to\infty} S_{\disc_m}(\varphi)=0$, for all $\varphi \in \cK$.
\end{definition}
\begin{definition}[Limit-conformity]\label{def:nlsignlim}
If $\disc$ is a gradient discretisation in the sense of Definition \ref{def:gd-nlsign}, define $W_{\mathcal{D}} : \{\bpsi\in C^2(\overline \O)^d\,:\,\bpsi\cdot\bfn=0\mbox{ on $\Gamma_2$}\} \to [0, +\infty)$ by
\begin{equation}
\begin{aligned}
&\forall \bpsi \in C^2(\overline \O)^d\mbox{ such that $\bpsi\cdot\bfn=0$ on $\Gamma_2$},\\
&W_{\mathcal{D}}(\bpsi)
 = \sup_{v\in X_{\mathcal{D},\Gamma_{2,3}}\setminus \{0\}}\frac{\Big|\dsp\int_{\Omega}(\nabla_{\mathcal{D}}v\cdot \bpsi + \Pi_{\mathcal{D}}v \div (\bpsi)) \ud x
 -\int_{\Gamma_3} \bpsi\cdot\bfn\mathbb{T}_{\mathcal{D}}v \ud x \Big|}{\| \nabla_\disc v \|_{L^p(\O)^d}}.
\label{conformityseepage}
\end{aligned}
\end{equation}
A sequence $(\mathcal{D}_{m})_{m \in \mathbb{N}}$ of gradient discretisations is \emph{limit-conforming} if, for all $\psi \in  C^2(\overline \O)^d$ such that $\bpsi\cdot\bfn=0$ on $\Gamma_2$,
 ${\lim_{m \rightarrow \infty}W_{\mathcal{D}_m}(\psi)=0}$.
\end{definition}

\begin{definition}[Compactness]\label{def:comnlobs}
A sequence $(\mathcal{D}_{m})_{m \in \mathbb{N}}$ of GDs is \emph{compact} if, for any sequence $(u_{m})_{m\in\mathbb{N}}$
with $u_m\in \cK_{\mathcal{D}_m}$ and such that $(\| \nabla_{\disc_m} u_m \|_{L^p(\O)^d})_{m \in \mathbb{N}}$ is bounded, the sequence $(\Pi_{\mathcal{D}_{m}} u_{m})_{m \in\mathbb{N}}$ is relatively compact in $L^{p}(\Omega)$.
\end{definition}



\subsection{Convergence results}
We can now state and prove our main convergence theorem for the gradient discretisation
method applied to the non-linear Signorini problem.

\begin{theorem}[Convergence of the GDM, non-linear Signorini problem]\label{thm:convnlsign}~\\
Under Assumptions \ref{hip:nosign}, let $(\mathcal{D}_{m})_{m \in \mathbb{N}}$ be a sequence of gradient discretisations in the sense of Definition \ref{def:gdnhseepage}, such that $(\disc_m)_{m\in\NN}$ is coercive, GD-consistent, limit-conforming and compact, and such that $\cK_{\mathcal{D}_{m}}$ is non-empty for any $m$. Then, for any $m \in \mathbb{N}$, the gradient scheme \eqref{gsnlsign} has at least one solution $u_{m} \in \cK_{\mathcal{D}_{m}}$.

Assume furthermore that\begin{equation}\label{newcons}
\begin{aligned}
&\exists \varphi_g \in W^{1,p}(\O)\mbox{ s.t. }\gamma(\varphi_g)=g\mbox{ and }\\
&\lim_{m\to\infty}\min \{ \| \Pi_{\disc_m}v-\varphi_g\|_{L^{p}(\O)} +
\|  \mathbb{T}_{\disc_m}v-\gamma(\varphi_g)\|_{L^{p}(\Gamma_{3})} \\
&\qquad+\|  \nabla_{\disc_m}v-\nabla\varphi_g\|_{L^{p}(\O)^d}\,:\, v-\mathcal{I}_{\disc_m,\Gamma_1} \gamma(\varphi_g) \in X_{\disc_m, \Gamma_{2,3}}\}=0.
\end{aligned}
\end{equation}
Then, up to a subsequence as $m\to\infty$, $\Pi_{\mathcal{D}_{m}}u_{m}$ converges strongly in $L^{p}(\Omega)$ to a weak solution $\bar{u}$ of Problem \eqref{wenlsign}, and $\nabla_{\disc_m}u_m$ converges weakly 
in $L^p(\O)^d$ to $\nabla\bar u$.

If moreover $\bfa$ is strictly monotonic in the sense
\begin{equation}\label{asm:smont}
\left.
\begin{array}{ll}
\left(\bfa(x,s,\xi)-\bfa(x,s,\chi)\right)\cdot \left(\xi-\chi \right) >0, \mbox{ for a.e. } x \in \O,\; \forall s \in \RR,\\
\forall \xi, \chi \in \RR^{d} \mbox{ with } \xi \neq \chi,
\end{array}
\right.
\end{equation}
then 
$\nabla_{\mathcal{D}_{m}}u_{m}$ converges strongly in $L^{p}(\Omega)^d$ to $\nabla \bar{u}$.
\label{th2}
\end{theorem}

\begin{remark} Assumption \eqref{newcons} is obviously always satisfied
if $g=0$ (take $\varphi_g=0$). For most sequences of gradient discretisations, the convergence stated in \eqref{newcons} actually
holds for any $\varphi$ with $g=\gamma(\varphi)$, and corresponds to the GD-consistency
of the method for non-homogeneous Fourier BCs (see \cite[Remark 3.49 and Definition 3.37]{S1}).
\end{remark}


\begin{proof}
The proof is inspired from  \cite{B1}, and follows the general path described
in \cite[Section 1.2]{D13} and \cite[Section 2.2]{CTAC14}.

\textbf{Step 1}: existence of a solution to the GS.

Let $\widetilde g\in \cK$ be a lifting of $g$, that is, such that $\gamma(\widetilde g)=g$. Introduce
\[
g_\disc=\argminB_{v\in \cK_\disc}\left(\| \Pi_\disc v-\widetilde g \|_{L^p(\O)}+ \| \nabla_\disc v-\nabla\widetilde g \|_{L^p(\O)^d}\right). 
\]
Let $\langle \cdot,\cdot \rangle$ be the the duality product between the finite dimensional space $X_{\disc,\Gamma_{2,3}}$ and its dual $X_{\disc,\Gamma_{2,3}}^{'}$. 
Define the operator $\mathcal A_{\disc}:X_{\disc,\Gamma_{2,3}} \to X_{\disc,\Gamma_{2,3}}^{'}$
by, for $\widehat u, \widehat v \in X_{\disc,\Gamma_{2,3}}$,
\[
\langle \mathcal A_{\disc}(\widehat u),\widehat v \rangle = \int_\O \bfa(x,\Pi_\disc (\widehat u+g_{\disc})(x), \nabla_\disc (\widehat u+g_{\disc})(x))\cdot\nabla_\disc (\widehat v+g_{\disc})(x) \ud x.
\]
Applying the same reasoning as in \cite{NB-1},
we check that $\mathcal A_\disc$ is an operator of the calculus of variations (this is
extremely easy here, due to the finite dimension of $X_{\disc,\Gamma_{2,3}}$).
The existence of a solution to the scheme \eqref{gsnlsign} is then a consequence of \cite[Theorem 8.2, Chap. 2]{NB-1} since, setting $\widehat u=u-g_\disc$ and $\widehat v=v-g_\disc$, this scheme can be re-written
\[
\mbox{find}\; \widehat u \in \cK_\disc -g_\disc \mbox{ such that }\forall \widehat v\in \cK_\disc -g_\disc,\;
\langle \mathcal A_{\disc}(\widehat u),\widehat u-\widehat v \rangle\le \ell(\widehat u-\widehat v),
\]
where $\ell\in X_{\disc,\Gamma_{2,3}}'$ is defined by $\ell(\widehat w)=\int_\O f\Pi_\disc \widehat w\ud x$.


\textbf{Step 2}: convergence towards the solution to the continuous model.

Let us start by estimating $\| \nabla_{\disc_m}u_m \|_{L^p(\O)^d}$. In \eqref{gsnlsign}, set $u:=u_{m}$, and $v:=v_{m}$ a generic element
in $\cK_{\mathcal{D}_{m}}$. By using the H\"older inequality and due to the coercivity assumption \eqref{asm:coer}, it follows that
\[
\begin{aligned}
\underline{a}\|\nabla_{\mathcal{D}_{m}}u_m\|_{L^{p}(\Omega)^{d}}^{p}
\leq{}&\int_\O \bfa(x,\Pi_{\disc_{m}}u_m,\nabla_{\disc_{m}}u_m)
\cdot\nabla_{\disc_{m}}u_m \ud x\\
\leq{}& \|f\|_{L^{p'}(\O)}\|\Pi_{\disc_{m}}(u_m-v_m)\|_{L^{p}(\O)}\\
&+\|\bfa(x,\Pi_{\disc_{m}}u_m,\nabla_{\disc_{m}}u_m)\|_{L^{p'}(\O)^d}\|\nabla_{\mathcal{D}_{m}} v_{m}\|_{L^{p}(\Omega)^{d}}.
\end{aligned}
\]
Since $u_{m}-v_{m}$ is an element in $X_{\disc_{m},\Gamma_{2,3}}$, applying the coercivity property (see Definition \ref{def:nlsigncoer}) gives $C_p$ not depending on $m$ such that $\|\Pi_{\disc_{m}}(u_m-v_m)\|_{L^{p}(\O)}
\le C_p\|\nabla_{\disc_{m}}(u_m-v_m)\|_{L^{p}(\O)^d}$. Thus, using the growth
assumption \eqref{asm:grow},
\[
\begin{aligned}
\underline{a}\|\nabla_{\mathcal{D}_{m}}u_m\|_{L^{p}(\Omega)^{d}}^{p}
\leq{}&C_p\|f\|_{L^{p'}(\O)}\left(\|\nabla_{\mathcal{D}_{m}} u_m\|_{L^p(\Omega)^{d}}+\|\nabla_{\mathcal{D}_{m}} v_{m}\|_{L^{p}(\Omega)^{d}}\right)\\
 &+\left(\|\overline{a}\|_{L^{p'}(\O)^d}+\mu \|\nabla_{\disc_{m}}u_m\|_{L^p(\O)^d}^{p-1}\right)
 \|\nabla_{\disc_{m}}v_m\|_{L^p(\O)^d}.
\end{aligned}
\]
Applying Young's inequality to this relation shows that
\begin{equation}\label{final.est.um}
\|\nabla_{\mathcal{D}_{m}}u_m\|_{L^{p}(\Omega)^{d}}^{p}
\leq C_1 \left(\|\nabla_{\disc_m}v_m\|_{L^p(\O)^d}^p
+\|f\|_{L^{p'}(\O)}^{p'}+\|\overline{a}\|_{L^{p'}(\O)^d}^{p'}\right)
\end{equation}
where $C_1$ does not depend on $m$. Let us now define, for $\varphi \in \cK$,
an element $P_{\disc_m} \varphi$ of $\cK_{{\disc_m}}$ by
\begin{equation}
P_{\disc_m}\varphi  ={ \displaystyle \argminB_{v \in \cK_{\disc_m}
}(\| \Pi_{\disc_m} v - \varphi \|_{L^{p}(\Omega)} + \| \nabla_{\disc_m} v - \nabla \varphi \|_{L^{p}(\Omega)^{d}})}.
\label{cons.funct}
\end{equation}
We have
\[
S_{\disc_m}(\varphi)=\| \Pi_{\disc_m} (P_{\disc_m}\varphi) - \varphi \|_{L^{p}(\Omega)} + 
\| \nabla_{\disc_m} (P_{\disc_m}\varphi) - \nabla \varphi \|_{L^{p}(\Omega)^{d}}.
\]
Set $v_m:=P_{\disc_m}\varphi$ in \eqref{final.est.um}.
By the triangle inequality
\[
\|\nabla_{\disc_m}v_m\|_{L^p(\O)^d}\leq S_{\disc_m}(\varphi)+\|\nabla\varphi\|_{L^p(\O)^d},
\]
and the GD-consistency of $\disc_m$ shows that $\|\nabla_{\disc_m}v_m\|_{L^p(\O)^d}$ is bounded.
Used in \eqref{final.est.um}, this proves that $\|\nabla_{\disc_m}u_m\|_{L^p(\O)^d}$ remains bounded.


Now, using \eqref{newcons}, \cite[Lemma 3.48]{S1} (slightly adjusted to the fact that the limit-conformity
involves here functions such that $\bpsi\cdot\bfn=0$ on $\Gamma_2$,
see \eqref{conformityseepage}) asserts the existence of $\bar u \in W^{1,p}(\Omega)$ and a subsequence, still denoted by $(\mathcal{D}_{m})_{m \in \mathbb{N}}$, such that $\gamma\bar u=g$ on $\Gamma_1$, $\Pi_{\mathcal{D}_{m}}u_{m}$ converges weakly to $\bar u$ in $L^{p}(\Omega)$, $\nabla_{\mathcal{D}_{m}}u_{m}$ converges weakly to $\nabla\bar u$ in $L^{p}(\Omega)^{d}$, and $\mathbb{T}_{\disc_m}u_m$ converges
weakly to $\gamma\bar u$ in $L^p(\Gamma_{3})$. 
Since $u_m \in \cK_{\mathcal{D}_{m}}$, 
we have $\mathbb T_{\disc_m}u_m \leq a$ on $\Gamma_3$, which implies $\gamma\bar u \leq a$
on $\Gamma_3$. In other words, $\bar u$ belongs to $\cK$. 
By the compactness hypothesis, the convergence of $\Pi_{\mathcal{D}_{m}}u_{m}$ to $\bar u$ is actually strong in $L^{p}(\Omega)$. Up to another subsequence, we can therefore assume that this
convergence holds almost everywhere on $\O$.

To complete this step, it remains to show that $\bar u$ is a solution to \eqref{wenlsign}. We
use the Minty trick. From assumption \eqref{asm:grow}, the sequence $\mathcal{A}_{\disc_m}=\bfa(x, \Pi_{\disc_{m}} u_m, \nabla_{\disc_{m}}u_m)$ is bounded in $L^{p'}(\O)^d$ and converges weakly up to a subsequence to some $\mathcal A$ in $L^{p'}(\O)^d$. Owing to the GD-consistency of the gradient discretisations, for all $\varphi \in \cK$ we have $\Pi_{\mathcal{D}_{m}}(P_{\mathcal{D}_{m}}\varphi) \to \varphi$ strongly in $L^{p}(\Omega)$ and $\nabla_{\mathcal{D}_{m}}(P_{\mathcal{D}_{m}}\varphi) \to \nabla\varphi$ strongly in $L^{p}(\Omega)^{d}$. Taking $v:=P_{\disc_{m}}\varphi$ as a test function in the gradient scheme \eqref{gsnlsign} and passing to the superior limit gives
\[
\begin{aligned}
\limsup_{m \rightarrow \infty} &\int_\O \bfa(x,\Pi_{\disc_{m}}u_m, \nabla_{\disc_{m}}u_m)\cdot\nabla_{\disc_{m}}u_m \ud x\\
\le{}&\limsup_{m \rightarrow \infty}\left(\int_\O f(\Pi_{\disc_m}u_m-\Pi_{\disc_m}P_{\disc_m}\varphi)\ud x + \int_\O \mathcal A_{\disc_m}\cdot \nabla_{\disc_m}P_{\disc_m}\varphi\ud x\right)\\
\leq{}&\int_\O f(\bar u -\varphi) \ud x
+ \int_\O \mathcal A\cdot\nabla\varphi \ud x, \quad \forall \varphi \in \mathcal{K}.
\end{aligned}
\]
Choosing $\varphi=\bar u$, yields
\begin{equation}\label{eq:conv1}
\limsup_{m \rightarrow \infty} \int_\O \bfa(x,\Pi_{\disc_{m}}u_m, \nabla_{\disc_{m}}u_m)\cdot\nabla_{\disc_{m}}u_m \ud x
\leq \int_\O \mathcal A\cdot\nabla\bar u \ud x.
\end{equation}
Using the monotonicity assumption \eqref{asm:mont}, one writes, for $\bG\in L^p(\O)^d$,
\begin{align}
\liminf_{m \rightarrow \infty}\Big[&\int_\O \bfa(x,\Pi_{\disc_{m}}u_m, \nabla_{\disc_{m}}u_m)\cdot\nabla_{\disc_{m}}u_m \ud x
- \int_\O \bfa(x,\Pi_{\disc_{m}}u_m, \nabla_{\disc_{m}}u_m)\cdot \bG \ud x\nonumber\\
&-\int_\O \bfa(x,\Pi_{\disc_{m}}u_m, \bG)\cdot\nabla_{\disc_{m}}u_m \ud x 
+\int_\O \bfa(x,\Pi_{\disc_{m}}u_m, \bG)\cdot \bG 
\ud x \Big]\nonumber\\
={}&\liminf_{m \rightarrow \infty}\int_\O \Big[\bfa(x,\Pi_{\disc_{m}}u_m, \nabla_{\disc_{m}}u_m)
-\bfa(x,\Pi_{\disc_{m}}u_m, \bG)\Big]\cdot \Big[\nabla_{\disc_{m}}u_m - \bG\Big] 
\ud x \nonumber\\
\geq{}& 0.
\label{minty.1}
\end{align}
The a.e.\ convergence of $\Pi_{\disc_m}u_m$, the growth property \eqref{asm:grow}
and the dominated convergence theorem show that $\bfa(x,\Pi_{\disc_m}u_m,\bG)
\to\bfa(x,\bar u,\bG)$ strongly in $L^{p'}(\O)^d$. Hence, passing to the limit in \eqref{minty.1},
\begin{align}\label{eq:conv5}
\liminf_{m \rightarrow \infty}&\int_\O \bfa(x,\Pi_{\disc_{m}}u_m, \nabla_{\disc_{m}}u_m)\cdot\nabla_{\disc_{m}}u_m \ud x 
- \int_\O \mathcal A\cdot\bG \ud x \nonumber\\
&-\int_\O \bfa(x,\bar u, \bG)\cdot\nabla\bar u \ud x 
+\int_\O \bfa(x, \bar u,\bG)\cdot \bG 
\ud x \geq 0.
\end{align}
Combining this inequality with \eqref{eq:conv1} yields
\begin{align*}
\int_\O \mathcal A\cdot\nabla \bar u \ud x 
-\int_\O \mathcal A\cdot \bG \ud x 
- \int_\O \bfa(x,\bar u, \bG)\cdot\nabla \bar u \ud x
+\int_\O \bfa(x,\bar u, \bG)\cdot \bG \ud x
\geq 0.
\end{align*}
Take $\bvarphi \in C_{c}^{\infty}(\O)^d$ and $\alpha >0$. Putting $\bG=\nabla\bar u + \alpha \bvarphi$ and dividing by $\alpha$, one obtains
\begin{equation*}
-\int_\O (\mathcal A-\bfa(x,\bar u, \nabla\bar u+\alpha\bvarphi))\cdot \bvarphi \ud x \geq 0, \; \forall \bvarphi \in C_{c}^{\infty}(\O)^d,\; \forall \alpha  >0.
\end{equation*}
Letting $\alpha \rightarrow 0$ and applying the dominated convergence theorem yields
\begin{equation*}
-\int_\O (\mathcal A-\bfa(x,\bar u, \nabla\bar u))\cdot \bvarphi \ud x \geq 0, \; \forall \bvarphi \in C_{c}^{\infty}(\O)^d.
\end{equation*}
Applied
to $-\bvarphi$ instead of $\bvarphi$, this leads to
\begin{equation*}
-\int_\O (\mathcal A-\bfa(x,\bar u, \nabla\bar u))\cdot \bvarphi \ud x=0, \; \forall \bvarphi \in C_{c}^{\infty}(\O)^d,
\end{equation*}
which implies that
\begin{equation}\label{eq:obs2}
\mathcal A=\bfa(x,\bar u, \nabla\bar u) \mbox{ a.e.\ on $\O$}.
\end{equation}
Setting $\bG=\nabla\bar u$ in \eqref{eq:conv5}, it follows that
\begin{equation}\label{eq:conv3}
\int_\O \bfa(x,\bar u, \nabla\bar u)\cdot \nabla\bar u \ud x 
\leq \liminf_{m \rightarrow \infty}\int_\O \bfa(x,\Pi_{\disc_{m}}u_m, \nabla_{\disc_{m}}u_m)\cdot \nabla_{\disc_{m}}u_m
 \ud x,
 \end{equation}
which gives, since $u_m$ is a solution to the gradient scheme \eqref{gsnlsign}, for all $\varphi \in \cK$,
\begin{multline*}
\int_\O \bfa(x,\bar u, \nabla\bar u)\cdot \nabla\bar u \ud x\\
\leq \liminf_{m \longrightarrow \infty}\Big[
\int_{\Omega}f \Pi_{\mathcal{D}_{m}} (u_m - P_{\mathcal{D}_{m}}\varphi) \ud x+
\int_{\Omega}\bfa    (x,\Pi_{\disc_{m}}u_m, \nabla_{\disc_{m}}u_m)
\cdot \nabla_{\mathcal{D}_{m}} (P_{\mathcal{D}_{m}}\varphi)\ud x\Big].
\end{multline*}
Using \eqref{eq:obs2} and the strong convergence of $\nabla_{\mathcal{D}_{m}} (P_{\mathcal{D}_{m}}\varphi)$ to $\nabla\varphi$ yields
\begin{eqnarray*}
\displaystyle{\int_\O \bfa(x,\bar u, \nabla\bar u)\cdot \nabla\bar u \ud x}
\leq\displaystyle\int_{\Omega}f (\bar u - \varphi) \ud x
+\int_{\Omega}\bfa(x,\bar u, \nabla\bar u)\cdot \nabla\varphi\ud x.
\end{eqnarray*}
This shows that $\bar u$ is a solution to \eqref{wenlsign}.   

\textbf{Step 3}: strong convergence of the gradients, if $\bfa$ is strictly monotonic.

Owing to \eqref{eq:conv1} and \eqref{eq:obs2},
\begin{equation}
\label{eq:conv6}
\limsup_{m \rightarrow \infty} \int_\O \bfa(x,\Pi_{\disc_{m}}u_m, \nabla_{\disc_{m}}u_m)\cdot\nabla_{\disc_{m}}u_m \ud x
\leq \int_\O \bfa(x,\bar u, \nabla\bar u)\cdot\nabla\bar u \ud x.
\end{equation}
Together with \eqref{eq:conv3}, we conclude that
\begin{equation}\label{eq:conv4}
\lim_{m \rightarrow \infty} \int_\O \bfa(x,\Pi_{\disc_{m}}u_m, \nabla_{\disc_{m}}u_m)\cdot\nabla_{\disc_{m}}u_m \ud x
= \int_\O \bfa(x,\bar u, \nabla\bar u)\cdot\nabla\bar u \ud x.
\end{equation}
The remaining reasoning to obtain the strong convergence of $\nabla_{\disc_m}u_m$ is exactly like in \cite{B1}. For the sake of completeness, we recall it. Equality \eqref{eq:conv4} leads to
\begin{eqnarray*}
\lim_{m \rightarrow \infty} \int_\O (\bfa(x,\Pi_{\disc_{m}}u_m, \nabla_{\disc_{m}}u_m)-\bfa(x,\bar u, \nabla\bar u))
\cdot(\nabla_{\disc_{m}}u_m - \nabla\bar u) \ud x =0.
\end{eqnarray*}
Making use of the fact that $(\bfa(x,\Pi_{\disc_{m}}u_m, \nabla_{\disc_{m}}u_m)-\bfa(x,\bar u, \nabla\bar u))
\cdot(\nabla_{\disc_{m}}u_m - \nabla\bar u) \geq 0$ a.e.\ on $\O$, we deduce that
\begin{eqnarray*}
(\bfa(x,\Pi_{\disc_{m}}u_m, \nabla_{\disc_{m}}u_m)-\bfa(x,\bar u, \nabla\bar u))
\cdot(\nabla_{\disc_{m}}u_m - \nabla\bar u) \rightarrow 0 \mbox{ in } L^1(\O).
\end{eqnarray*}
Up a subsequence, the convergence holds almost everywhere. The strict monotoni\-city assumption \eqref{asm:smont} and \cite[Lemma 3.2]{B1} yield $\nabla_{\disc_{m}}u_m \rightarrow \nabla \bar u$ a.e.\ as $m \rightarrow \infty$. Furthermore, as a consequence, $\bfa(x,\Pi_{\disc_{m}}u_m, \nabla_{\disc_{m}}u_m)\cdot \nabla_{\disc_{m}}u_m \rightarrow \bfa(x,\bar u, \nabla\bar u)\cdot \nabla\bar u$ a.e. Since $\bfa(x,\Pi_{\disc_{m}}u_m, \nabla_{\disc_{m}}u_m)\cdot \nabla_{\disc_{m}}u_m \geq 0$, and taking into account  \eqref{eq:conv4}, \cite[Lemma 3.3]{B1} gives the strong convergence of $\bfa(x,\Pi_{\disc_{m}}u_m, \nabla_{\disc_{m}}u_m)\cdot \nabla_{\disc_{m}}u_m$ to $\bfa(x,\bar u, \nabla\bar u)\cdot \nabla\bar u$ in $L^1(\O)$ as $m \rightarrow \infty$. As a consequence of this $L^1$-convergence, we obtain the equi-integrability of the sequence of functions $\bfa(x,\Pi_{\disc_{m}}u_m, \nabla_{\disc_{m}}u_m)\cdot \nabla_{\disc_{m}}u_m$. This provides, with \eqref{asm:coer}, the equi-integrability of $(|\nabla_{\disc_{m}}u_m|^p)_{m\in \NN}$. The strong convergence of $\nabla_{\disc_{m}}u_m$ to $\nabla\bar u$ in $L^p(\O)$ is then directly implied by the Vitali theorem.

\end{proof}



\section{Obstacle problem and generalised Bulkley fluid models}\label{sec:obs}

\subsection{Continuous problems}

\subsubsection{Obstacle problem}

We are concerned here with other kinds of variational inequalities. The first one is an obstacle model, in which the inequalities are imposed inside the domain $\O$. It is formulated as
\begin{subequations}\label{nlobs}
\begin{align}
(\div{\bf a}(x,\bar u, \nabla\bar u)+f)(\psi-\bar{u}) &= 0 \mbox{\quad in $\Omega$,} \label{nlobs1}\\
-\div{\bf a}(x,\bar u, \nabla\bar u) &\leq f \mbox{\quad in $\Omega$,} \label{nlobs2}\\
\bar{u}&\leq \psi \mbox{\quad in $\Omega$,} \label{nlobs3}\\
\bar{u} &= h \mbox{\quad on $\partial\Omega$.} \label{nlobs4}
\end{align}
\end{subequations}
Let us provide the assumptions on the data of this model.
\begin{assumptions}\label{hyp-nl-obs}
\begin{enumerate}
\item the operator $\bfa$ and the domain $\O$ satisfy the same pro\-perties as in Assumption \ref{hip:nosign},
\item the function $f$ belongs to $L^{p'}(\O)$, the boundary function $h$ is in $W^{1-\frac{1}{p},p}(\partial
\O)$ and the obstacle function $\psi$ belongs to $L^p(\O)$,
\item the closed convex set $\mathcal{K}:=\{ v \in W^{1,p}(\Omega)\; : \; v \leq \psi\; \mbox{in}\; \Omega,\; \gamma(v)=h \mbox{ on } \partial\O\}$ is non-empty.
\end{enumerate}
\end{assumptions}
The weak formulation of the obstacle problem \eqref{nlobs} is
\begin{equation}\label{wenlobs}
\left\{
\begin{array}{ll}
\dsp \mbox{Find } \bar{u} \in \mathcal{K}\mbox{ such that, $\forall v\in \mathcal{K}$,}\\
\dsp\int_\O \bfa(x,\bar u,\nabla\bar u) \cdot \nabla(\bar{u}-v)\ud x
\leq \int_\O f(\bar{u}-v)\ud x.
\end{array}
\right.
\end{equation}

\subsubsection{Generalised Bulkley model}

The second problem is called the Bulkley model, whose weak formulation is given by 
\begin{equation}\label{webulk}
\left\{
\begin{array}{ll}
\mbox{Find}\; \bar{u} \in W_{0}^{1,p}(\Omega) \mbox{ such that, for all  } v \in W_{0}^{1,p}(\Omega),
\\
\dsp\int_\O \bfa(x,\bar u,\nabla\bar u)\cdot \nabla(\bar{u}-v)\ud x
+\int_\O |\nabla\bar u| \ud x - \int_\O |\nabla v| \ud x \\
\qquad\leq \dsp\int_\O f(\bar{u}-v)\ud x.
\end{array}
\right.
\end{equation}
Here the operator $\bfa$ is assumed to satisfy \eqref{asm:carth}--\eqref{asm:mont} and the domain $\O$ has a Lipschitz boundary.
Models considered in the removal of materials from a duct by using fluids \cite{A3} are included in \eqref{webulk} by setting $\bfa(x,\bar u,\nabla\bar u)=|\nabla\bar u|^{p-2}\nabla\bar u$. 

\medskip

As for the Signorini problem, \cite[Theorem 8.2, Chap. 2]{NB-1} yields the existence of a solution to each of the problems \eqref{wenlobs} and \eqref{webulk}.



\subsection{Discrete problems}\label{sec:disc-pbls-obs}

\subsubsection{Obstacle problem}

Let us recall the definition of a gradient discretisation for non-homogeneous Dirichlet
boundary conditions \cite{S1}.

\begin{definition}[GD for non-homogeneous Dirichlet boundary conditions]\label{def:gdnlobs} A gradient discretisation $\mathcal{D}$ for non-homogeneous Dirichlet boundary conditions is defined by $\mathcal{D}=(X_\disc, \Pi_\disc,\mathcal{I}_{\disc,\partial\O},\nabla_{\mathcal{D}})$, where:
\begin{enumerate}
\item the set of discrete unknowns $X_\disc=X_{\disc,0} \oplus X_{\disc,\partial\O}$ is a direct sum of two finite dimensional spaces on $\RR$, representing respectively the interior degrees of freedom and the boundary degrees of freedom,
\item the linear mapping $\Pi_{\mathcal{D}} : X_\disc \rightarrow L^{p}(\varOmega)$ provides the reconstructed function,
\item the linear mapping $\cI_{\disc,\partial\O} : W^{1-\frac{1}{p},p}(\partial\O) \rightarrow X_{\disc,\partial\O}$ provides an interpolation operator for the trace of functions in $W^{1,p}(\O)$,
\item the linear mapping $\nabla_{\mathcal{D}} : X_\disc \rightarrow L^{p}(\varOmega)^{d}$ gives a reconstructed gradient, which must be defined such that $\| \nabla_\disc \cdot \|_{L^p(\O)^d}$ is a norm on $X_{\disc,0}$.
\end{enumerate}
\end{definition}

\begin{definition}[GS for the non-linear obstacle problem] Let $\mathcal{D}$ be a gradient discretisation in the sense of Definition \ref{def:gdnlobs}. The corresponding gradient scheme for (\ref{wenlobs}) is given by
\begin{equation}\label{gsnlobs}
\left\{
\begin{array}{ll}
\dsp \mbox{Find}\; u \in \mathcal{K}_{\mathcal{D}}:=\{ v \in X_{\mathcal{D},0}+\cI_{\disc,\partial\O}h\; : \; \Pi_{\mathcal{D}} v \leq \psi\; \mbox{in}\; \Omega\}\; 
\mbox{ s.t., $\forall v\in \mathcal{K}_{\mathcal{D}}$},\\
\dsp\int_\O {\bf a}(x,\Pi_\disc u, \nabla_\disc u) \cdot \nabla_{\mathcal{D}}(u-v)\ud x
\leq \int_\O f\Pi_{\mathcal{D}}{(u-v)}\ud x.
\end{array}
\right.
\end{equation}
\end{definition}

\subsubsection{Generalised Bulkley model}

\begin{definition}[GD for homogeneous Dirichlet boundary conditions]
\label{def:gdbulk}
A gradient discretisation $\mathcal{D}$ for homogeneous Dirichlet boundary conditions is defined by $\mathcal{D}=(X_{\disc,0}, \Pi_\disc,\nabla_{\mathcal{D}})$, where $X_{\disc,0}$ is a finite dimensional vector space over $\RR$, taking into account the zero boundary condition, and $\Pi_\disc$ and $\nabla_\disc$ are as in Definition \ref{def:gdnlobs} but defined on $X_{\disc,0}$.
\end{definition}

\begin{definition}[GS for the Bulkley model] Let $\mathcal{D}$ be a gradient discretisation in the sense of Definition \ref{def:gdbulk}. The corresponding gradient scheme for \eqref{webulk} is given by
\begin{equation}\label{gsbulk}
\left\{
\begin{array}{ll}
\mbox{Find}\; \bar{u} \in X_{\disc,0} \mbox{ such that for all  } v \in X_{\disc,0},
\\
\dsp\int_\O \bfa(x,\Pi_\disc u,\nabla_\disc u)\cdot \nabla_\disc(u-v)\ud x
+\int_\O |\nabla_\disc u | \ud x - \int_\O |\nabla_\disc v| \ud x \\
\qquad\leq \dsp\int_\O f\Pi_\disc(u-v)\ud x.
\end{array}
\right.
\end{equation}

\end{definition}


\subsubsection{Properties of GDs}

Except for the restriction to the convex sets $\cK$ and $\cK_\disc$ in the
GD-consistency, all the properties of GDs required for the convergence analysis
of the GDM on the non-linear obstacle and Bulkley models are similar
to the corresponding ones for GDs adapted to PDEs \cite{B1,S1}.

\begin{definition}[Coercivity]\label{def:obscoercivity}
If $\disc$ is a gradient discretisation in the sense of Definition \ref{def:gdnlobs} or Definition \ref{def:gdbulk}, define
\begin{equation}
C_{\mathcal{D}} =  {\displaystyle \max_{v \in X_{\mathcal{D},0}\setminus\{0\}}\frac{\|\Pi_{\mathcal{D}}v\|_{L^{p}(\Omega)}}{\|\nabla_{\mathcal{D}} v\|_{L^{p}(\Omega)^{d}}}}.
\label{coercivityobs}
\end{equation} 
A sequence $(\mathcal{D}_{m})_{m \in \mathbb{N}}$ of such gradient discretisations is \emph{coercive} if $(C_{\mathcal D_m})_{m\in\NN}$ remains bounded.
\end{definition}
\begin{definition}[GD-Consistency]\label{def:obsconst}
If $\disc$ is a gradient discretisation in the sense of Definition \ref{def:gdnlobs}, let
$S_{\mathcal{D}} : \cK \to [0, +\infty)$ be defined by
\begin{equation}
\forall \varphi \in \mathcal{K}, \; S_{\mathcal{D}}(\varphi)=
\min_{v \in \cK_{\disc}}\left( \| \Pi_{\mathcal{D}} v - \varphi \|_{L^{p}(\Omega)} 
+ \| \nabla_{\mathcal{D}} v - \nabla \varphi \|_{L^{p}(\Omega)^{d}}\right).
\label{consistencyobs}
\end{equation}
If $\disc$ is a gradient discretisation in the sense of Definition \ref{def:gdbulk},
$S_\disc$ is defined the same way with $(\cK,\cK_\disc)$ replaced by $(W^{1,p}_0(\O),
X_{\disc,0})$.

A sequence $(\mathcal{D}_{m})_{m \in \mathbb{N}}$ of such gradient discretisations is \emph{GD-consistent} if for all $\varphi \in \cK$, ${\lim_{m \rightarrow \infty}S_{\mathcal{D}_m}(\varphi)=0}$.
\end{definition}
\begin{definition}[Limit-conformity]\label{def:obslimconf}
If $\disc$ is a gradient discretisation in the sense of Definition \ref{def:gdnlobs} or Definition \ref{def:gdbulk}, define
$W_{\mathcal{D}} : C^2(\overline \O)^d \to [0, +\infty)$ by
\begin{equation}
\begin{aligned}
&\forall \bpsi \in C^2(\overline{\O})^d,\\
&W_{\mathcal{D}}(\bpsi)=
 \sup_{v\in X_{\mathcal{D},0}\setminus \{0\}}\frac{1}{\|\nabla_{\mathcal{D}} v\|_{L^{p}(\Omega)^{d}}} \Big|\int_{\Omega}(\nabla_{\mathcal{D}}v\cdot \bpsi + \Pi_{\mathcal{D}}v \div (\bpsi)) \ud x
\Big|.
\end{aligned}
\label{conformityobs}
\end{equation}
A sequence $(\mathcal{D}_{m})_{m \in \mathbb{N}}$ of such gradient discretisations
is \emph{limit-conforming} if for all $\psi \in  C^2(\overline{\O})^d$, ${\lim_{m \rightarrow \infty}W_{\mathcal{D}_m}(\psi)=0}$.

\end{definition}
Finally, Definition \ref{def:comnlobs} (compactness) remains the same for gradient discretisations in the sense of Definition \ref{def:gdnlobs} or Definition \ref{def:gdbulk},
with $\cK_{\disc_m}$ replaced by $X_{\disc_m,0}$ in the latter case.



\subsection{Convergence results}

The following two theorems state the convergence properties of the GDM
for the non-linear obstacle problem and the Bulkley model. Note that for quasi-linear operators (that is, $\bfa(x,\bar u,\nabla \bar u)=\Lambda(x,\bar u)\nabla\bar u$),
the convergence of the GDM for the obstacle problem was established in \cite{Y}.

\begin{theorem}[Convergence of the GDM, non-linear obstacle problem]\label{thm:convnlobs}~\\
Un\-der Assumptions \ref{hyp-nl-obs}, let $(\mathcal{D}_{m})_{m \in \mathbb{N}}$ be a sequence of gradient discretisations in the sense of Definition \ref{def:gdnlobs}, such that $(\disc_m)_{m\in\NN}$ is coercive, GD-consistent, limit-conforming and compact, and such that $\cK_{\mathcal{D}_{m}}$ is a non-empty set for any $m$.

Then, for any $m \in \mathbb{N}$, the gradient scheme \eqref{gsnlobs} has at least one solution $u_{m} \in \cK_{\mathcal{D}_{m}}$ and, up to a subsequence as $m\to\infty$, $\Pi_{\mathcal{D}_{m}}u_{m}$ converges strongly in $L^{p}(\O)$ to a weak solution $\bar{u}$ of Problem \eqref{wenlobs} and $\nabla_{\disc_m}u_m$ converges weakly in $L^p(\O)^d$ to $\nabla\bar u$. 

If the strict monotonicity \eqref{asm:smont} is assumed, then $\nabla_{\mathcal{D}_{m}}u_{m}$ converges strongly in $L^{p}(\O)^d$ to $\nabla \bar{u}$.
\end{theorem}

\begin{theorem}[Convergence of the GDM, Bulkley model]\label{thm:convbulkley}~\\
Un\-der As\-sumptions \eqref{asm:carth}--\eqref{asm:mont} and $f \in L^{p'}(\O)$, let $(\mathcal{D}_{m})_{m \in \mathbb{N}}$ be a sequence of gradient discretisations in the sense of Definition \ref{def:gdbulk}, such that $(\disc_m)_{m\in\NN}$ is coercive, GD-consistent, limit-conforming and compact.

Then, for any $m \in \mathbb{N}$, the gradient scheme \eqref{gsbulk} has at least one solution $u_{m} \in X_{\disc,0}$ and, up to a subsequence as $m\to\infty$, $\Pi_{\mathcal{D}_{m}}u_{m}$ converges strongly in $L^{p}(\Omega)$ to a weak solution $\bar{u}$ of Problem \eqref{webulk} and $\nabla_{\disc_m}u_m$ converges weakly to $\nabla\bar u$ in $L^p(\O)^d$. 

If we also assume that $\bfa$ is strictly monotonic in the sense of \eqref{asm:smont}, then 
$\nabla_{\mathcal{D}_{m}}u_{m}$  converges strongly in $L^{p}(\Omega)^d$ to $\nabla \bar{u}$.  
\end{theorem}

The proof of Theorem \ref{thm:convnlobs} is extremely similar to the proof of 
Theorem \ref{thm:convnlsign}. We therefore only provide the proof of Theorem \ref{thm:convbulkley}.

\begin{proof}[Proof of Theorem \ref{thm:convbulkley}]
Let us define the operator $\mathcal A_\disc:X_{\disc,0}\to X_{\disc,0}^{'}$ and the functional $J_\disc:X_{\disc,0} \to \RR^+$ as follows:
\[
\begin{aligned}
&\langle \mathcal A_\disc (u),v \rangle= \int_\O \bfa(x,\Pi_\disc u(x),\nabla_\disc u(x))\cdot\nabla_\disc v(x)  \ud x \quad \mbox{and}\\  
&J_\disc(u)=\int_\O |\nabla_\disc u(x)| \ud x,
\end{aligned}
\]
where $\langle\cdot,\cdot\rangle$ is the duality product between $X_{\disc,0}^{'}$ and $X_{\disc,0}$.
Applying the same arguments as in \cite{NB-1}, one can easily prove that the operator $\mathcal A_\disc$ is pseudo-monotone and obtain, since $J_\disc\ge 0$,
\begin{align*}
\frac{\left\langle \mathcal A_\disc(u), u - \phi \right\rangle +J_\disc(u)}{\|\nabla_\disc u\|_{L^p(\O)^d}} \to +\infty \quad \mbox{as}\quad\|\nabla_\disc u\|_{L^p(\O)^d}\to \infty.
\end{align*}
A direct application of \cite[Theorem 8.5, Chap. 2]{NB-1} then gives the existence of a solution to Problem \eqref{gsbulk}. 
 
We now show that $\| \nabla_{\disc_m}u_m \|_{L^p(\O)^d}$ is bounded. Choose $u:=u_{m}$, and $v:=0\in X_{\disc_{m},0}$ in \eqref{gsbulk}. Due to the coercivity assumption \eqref{asm:coer} on $\bfa$, 
the H\"older inequality and the coercivity of $(\disc_m)_{m\in\NN}$,
one has
\begin{align*}
\underline{a}\|\nabla_{\mathcal{D}_{m}}u_m\|_{L^{p}(\Omega)^{d}}^{p}
\leq{}&\dsp\int_\O \bfa(x,\Pi_{\disc_{m}}u_m,\nabla_{\disc_{m}}u_m)
\cdot\nabla_{\disc_{m}}u_m \ud x\\
\leq{}& \dsp\int_\O f \Pi_{\disc_{m}}u_m \ud x\\
\le{}& C_p\|f\|_{L^{p'}(\O)}\|\nabla_{\mathcal{D}_{m}}u_m\|_{L^{p}(\Omega)^{d}}.
\end{align*}
This shows that $\|\nabla_{\mathcal{D}_{m}}u_m\|_{L^{p}(\Omega)^{d}}$ is bounded.
According to \cite[Lemma 2.15]{S1}, there exists $\bar u \in W_0^{1,p}(\Omega)$ and a subsequence, denoted in the same way, such that 
$\Pi_{\mathcal{D}_{m}}u_{m}$ converges weakly to $\bar u$ in $L^{p}(\Omega)$ and $\nabla_{\mathcal{D}_{m}}u_{m}$ converges weakly to $\nabla\bar u$ in $L^{p}(\Omega)^{d}$. In fact, the strong convergence of the sequence $\Pi_{\mathcal{D}_{m}}u_{m}$ to $\bar u$ in $L^p(\O)$ is ensured by the compactness property. The growth assumption \eqref{asm:grow} shows that the sequence $\mathcal{A}_{\disc_m}=\bfa(x, \Pi_{\disc_{m}} u_m, \nabla_{\disc_{m}}u_m)$ is bounded in $L^{p'}(\O)^d$ and
thus, up to a subsequence, that it converges weakly to some $\mathcal A$ in this space. 

Define $P_{\disc_m}$ as in \eqref{cons.funct} with $\cK$ and $\cK_{\disc_m}$
replaced with $W^{1,p}_0(\O)$ and $X_{\disc_m,0}$, respectively. The consistency 
 guarantees that $\Pi_{\mathcal{D}_{m}}(P_{\mathcal{D}_{m}}\varphi) \to \varphi$ strongly in $L^{p}(\Omega)$ and $\nabla_{\mathcal{D}_{m}}(P_{\mathcal{D}_{m}}\varphi) \to \nabla\varphi$ strongly in $L^{p}(\Omega)^{d}$, for all $\varphi\in W_0^{1,p}$. 
Inserting $v:=P_{\disc_{m}}\varphi$ into the gradient scheme \eqref{gsnlobs}, we obtain
\begin{equation}\label{eq:bulk5}
\begin{aligned}
\int_\O \bfa(x,\Pi_{\disc_m}u_m,\nabla_{\disc_m}u_m) &\cdot\nabla_{\disc_{m}}u_m\ud x\\
\leq{}&
\int_\O\bfa(x,\Pi_{\disc_m}u_m,\nabla_{\disc_m}u_m)\cdot\nabla_{\disc_{m}}P_{\disc_{m}}\varphi\ud x\\
&+
\int_\O f\Pi_{\disc_{m}}(u_m-P_{\disc_{m}}\varphi)\ud x\\
&-
\int_\O |\nabla_{\disc_{m}} u_m | \ud x 
+ \int_\O |\nabla_{\disc_{m}} (P_{\disc_{m}}\varphi)| \ud x.
\end{aligned}
\end{equation} 
All the terms except the last two can be handled as in Theorem \ref{thm:convnlsign}. From the strong convergence of $\nabla_{\disc_m}(P_{\disc_{m}}\varphi)$, letting $m \rightarrow \infty$ in the last term implies
\begin{equation}\label{eq:str1}
\lim_{m \rightarrow \infty}\int_\O |\nabla_{\disc_{m}} (P_{\disc_{m}}\varphi)| \ud x=
\int_\O |\nabla\varphi| \ud x.
\end{equation}
Estimating $\liminf_{m\to\infty}\int_\O |\nabla_{\disc_m}u_m|\ud x$ is rather
standard. 
For any $\mathbf{w}\in L^\infty(\O)^d$ such that $|\mathbf{w}|\leq 1$, 
write
$\int_\O \mathbf{w}\cdot\nabla_{\disc_m}u_m \ud x \leq \int_\O |\nabla_{\disc_m}u_m| \ud x$.
The weak convergence in $L^p(\O)^d$ of $\nabla_{\disc_m}u_m$ then yields
\[
\int_\O \mathbf{w}\cdot\nabla\bar u \ud x=
\lim_{m\to\infty}\int_\O \mathbf{w}\cdot\nabla_{\disc_m}u_m \ud x
\le \liminf_{m\to\infty}\int_\O |\nabla_{\disc_m}u_m| \ud x.
\] 
Taking the supremum over $\mathbf{w}$ leads to
\[
\int_\O |\nabla\bar u| \ud x \leq \liminf_{m\to\infty}\int_\O |\nabla_{\disc_m}u_m| \ud x. 
\] 
From this estimation and \eqref{eq:str1}, passing to the superior limit in \eqref{eq:bulk5} gives
\begin{multline}\label{ineq:limbulk}
\limsup_{m\to \infty}
\int_\O \bfa(x,\Pi_{\disc_m}u_m,\nabla_{\disc_m}u_m) \cdot\nabla_{\disc_{m}}u_m\ud x\\
\leq{}
\int_\O \mathcal A\cdot\nabla\varphi\ud x
+\int_\O f(\bar u-\varphi)\ud x-\int_\O |\nabla\bar u | \ud x 
+ \dsp\int_\O |\nabla\varphi| \ud x.
\end{multline} 
Since this inequality holds for any $\varphi \in W_0^{1,p}(\O)$, making $\varphi=\bar u$
gives
\begin{equation}\label{eq:bulk2}
\limsup_{m\to \infty}\dsp\int_\O \bfa(x,\Pi_{\disc_m}u_m,\nabla_{\disc_m}u_m)\cdot \nabla_{\disc_{m}}u_m\ud x 
\leq \int_\O \mathcal A\cdot\nabla \varphi \ud x.
\end{equation}
Exactly as Theorem \ref{thm:convnlsign}, it is then shown that $\mathcal A=\bfa(x,\bar u,\nabla\bar u)$
and 
\begin{equation*}
\int_\O \bfa(x,\bar u, \nabla\bar u)\cdot \nabla\bar u \ud x 
\leq \liminf_{m \rightarrow \infty}\int_\O \bfa(x,\Pi_{\disc_{m}}u_m, \nabla_{\disc_{m}}u_m)\cdot \nabla_{\disc_{m}}u_m
 \ud x.
\end{equation*}
Substituting $\mathcal A$ and using this relation in \eqref{ineq:limbulk} show that $\bar u$ is a solution to Problem \eqref{webulk}.
The rest of proof follows the same lines as for Theorem \ref{thm:convnlsign}.
\end{proof}

\section{Approximate barriers}\label{sec:approx.barriers}

Let us now discuss the case of approximate barriers. In most numerical methods, such as the $\mathbb P_1$ finite elements for instance, the standard interpolant of a smooth function $v$ is constructed by taking the value of $v$ at interpolation nodes. When $v$ is bounded by the barrier ($a$ for the Signorini problem, $\psi$ for the obstacle problem), this interpolation may not satisfy the barriers conditions at all points on the boundary/in the domain, especially in the case of non-constant barriers. It is therefore classical to modify these barriers conditions when discretising the model. This modification can often be written in the following way. 

Using $a_\disc\in L^p(\partial\O)$ (for the Signorini problem) or $\psi_\disc \in L^p(\O)$
(for the obstacle problem), which are respectively approximations of $a$ or $\psi$, we introduce the convex sets
\[\cK_{\disc,a_\disc}:=\{ v \in \cI_{\disc,\Gamma_1}g+X_{\disc,\Gamma_{2,3}}\; : \; \mathbb{T}_{\mathcal{D}}v \leq a_\disc\; \mbox{on}\; \Gamma_{3}\}
\]
or
\[
\cK_{\disc,\psi_\disc}:=\{ v \in \cI_{\disc,\Gamma_1}h+X_{\disc,0}\; : \; \Pi_\disc v \leq \psi_\disc\}.
\]
The schemes \eqref{gsnlsign} or \eqref{gsnlobs} are then modified by replacing the set $\cK_\disc$ by $\cK_{\disc,a_\disc}$ in the Signorini case, or by $\cK_{\disc,\psi_\disc}$ in the obstacle case. The convergence results for this case of approximate barriers are given in the following theorems, whose proofs are identical to that of Theorem \ref{thm:convnlsign} (see \cite[Section 6]{AD14} for the case
of approximate barriers in gradient schemes for linear VIs).

\begin{theorem}[Convergence: non-linear Signorini, approximate barrier]\label{thm:bari-conv-nlsign}~\\
Under the assumptions of Theorem \ref{thm:convnlsign}, let $(\mathcal{D}_{m})_{m \in \mathbb{N}}$ be a sequence of gradient discretisations in the sense of Definition \ref{def:gdnhseepage}, such that $(\disc_m)_{m\in\NN}$ is coercive, limit-conforming, compact, and GD-consistent (with $S_\disc$ defined using
$\cK_{\disc,a_\disc}$ instead of $\cK_{\disc}$). Assume that each $\cK_{\disc_m,a_{\disc_m}}$ is non-empty. 

Then, for any $m \in \mathbb{N}$, there exists at least one solution $u_{m} \in \cK_{\disc_m,a_{\disc_m}}$ to the gradient scheme \eqref{gsnlsign} in which $\cK_{\disc_m}$ has been replaced with $\cK_{\disc_m,a_{\disc_m}}$. If moreover $a_{\disc_m}\to a$ in $L^p(\partial\O)$ as $m\to\infty$, then the convergences of $\Pi_{\disc_m}u_m$ and $\nabla_{\disc_m}u_m$ stated in Theorem \ref{thm:convnlsign} still hold. 
\end{theorem} 

\begin{theorem}[Convergence: non-linear obstacle problem, approximate barrier]\label{thm:bari-conv-nlobs}~\\
Under the assumptions of Theorem \ref{thm:convnlobs}, let $(\mathcal{D}_{m})_{m \in \mathbb{N}}$ be a sequence of gradient discretisations in the sense of Definition \ref{def:gdnlobs}, such that $(\disc_m)_{m\in\NN}$ is coercive, limit-conforming, compact, and GD-consistent (with $S_\disc$ defined using $\cK_{\disc,\psi_\disc}$ instead of $\cK_{\disc}$). Assume that $\cK_{\disc_m,\psi_{\disc_m}}$ is non-empty for any $m$.

Then, for any $m \in \mathbb{N}$, there exists at least one solution $u_{m} \in \cK_{\disc_m,\psi_{\disc_m}}$ to the gradient scheme \eqref{gsnlobs} in which $\cK_{\disc_m}$ has been replaced with $\cK_{\disc_m,\psi_{\disc_m}}$. Furthermore, if $\psi_{\disc_m}\to \psi$ in $L^p(\O)$ as $m\to\infty$, then the convergences of $\Pi_{\disc_m}u_m$ and $\nabla_{\disc_m}u_m$ given in Theorem \ref{thm:convnlobs} still hold.   
\end{theorem}


\section{Application to the hybrid mimetic mixed methods}\label{sec:HMM}

The gradient discretisation method is used here to design a hybrid mimetic mixed (HMM)
scheme for non-linear variational inequalities. It is shown in \cite{B5} that the HMM method gathers three different families: the hybrid finite volume method, the (mixed-hybrid) mimetic finite differences methods, and the mixed finite volume methods. The mimetic methods have become efficient tools to discretise heterogeneous anisotropic diffusion problems on generic meshes. In \cite{AD14} we established the HMM method for the linear Signorini and obstacle problems (i.e., $\bfa(x,\bar u,\nabla\bar u)=\Lambda(x)\nabla\bar u$). 
The only other application of mimetic method to variational inequalities, of which we are aware,
concerns linear variational inequalities and the nodal mimetic finite difference
method \cite{A-32}. The HMM scheme described here for non-linear variational
inequalities seems to be the first numerical scheme for these models on generic meshes.

Let us first recall the notion of polytopal mesh \cite{B1}.

\begin{definition}[Polytopal mesh]\label{def:polymesh}~
Let $\Omega$ be a bounded polytopal open subset of $\RR^d$ ($d\ge 1$). 
A polytopal mesh of $\O$ is given by $\polyd = (\mesh,\edges,\centers)$, where:
\begin{enumerate}
\item $\mesh$ is a finite family of non empty connected polytopal open disjoint subsets of $\O$ (the cells) such that $\overline{\O}= \dsp{\cup_{K \in \mesh} \overline{K}}$.
For any $K\in\mesh$, $|K|>0$ is the measure of $K$ and $h_K$ denotes the diameter of $K$.

\item $\edges$ is a finite family of disjoint subsets of $\overline{\O}$ (the edges of the mesh in 2D,
the faces in 3D), such that any $\edge\in\edges$ is a non empty open subset of a hyperplane of $\RR^d$ and $\edge\subset \overline{\O}$.
We assume that for all $K \in \mesh$ there exists  a subset $\edgescv$ of $\edges$
such that $\dr K  = \dsp{\cup_{\edge \in \edgescv}} \overline{\edge}$. 
We then set $\mesh_\edge = \{K\in\mesh\,:\,\edge\in\edgescv\}$
and assume that, for all $\edge\in\edges$, $\mesh_\edge$ has exactly one element
and $\edge\subset\partial\O$, or $\mesh_\edge$ has two elements and
$\edge\subset\O$. 
$\edgesint$ is the set of all interior faces, i.e. $\edge\in\edges$ such that $\edge\subset \O$, and $\edgesext$ the set of boundary
faces, i.e. $\edge\in\edges$ such that $\edge\subset \dr\O$.
For $\edge\in\edges$, the $(d-1)$-dimensional measure of $\edge$ is $|\edge|$,
the centre of mass of $\edge$ is $\centeredge$, and the diameter of $\edge$ is $h_\edge$.

\item $\centers = (x_K)_{K \in \mesh}$ is a family of points of $\O$ indexed by $\mesh$ and such that, for all  $K\in\mesh$,  $\xcv\in K$ ($\xcv$ is sometimes called the ``centre'' of $\cv$). 
We then assume that all cells $K\in\mesh$ are  strictly $\xcv$-star-shaped, meaning that 
if $x\in \overline{K}$ then the line segment $[\xcv,x)$ is included in $K$.
\end{enumerate}
For a given $K\in \mesh$, let $\bfn_{K,\sigma}$ be the unit vector normal to $\sigma$ outward to $K$
and denote by $d_{K,\sigma}$ the orthogonal distance between $x_K$ and $\sigma\in\mathcal E_K$.
The size of the discretisation is $h_\mesh=\sup\{h_K\,:\; K\in \mesh\}$.
\end{definition}

\begin{remark}This definition allows for very generic meshes, possibly non-conforming (with `hanging nodes') and with non-convex cells. In particular, all meshes used in Section \ref{sec:test}
satisfy this definition.
\end{remark}

\subsection{HMM for the Signorini problem}\label{sec:HMM-nlsign}

Let $\polyd$ be a polytopal mesh that is aligned with the
boundaries $(\Gamma_i)_{i=1,2,3}$, that is, for any $i=1,2,3$,
each boundary edge is either fully included in $\Gamma_i$ or disjoint from this set.
We describe here a gradient discretisation that corresponds, for linear
diffusion problems and standard boundary conditions, to the HMM method
\cite{B1,S1}.

Define two discrete spaces as follows:  
\begin{equation}\label{def.XDG}
\begin{aligned}
X_{\disc,\Gamma_{2,3}}=\{ &v=((v_{K})_{K\in \mathcal{M}}, (v_{\sigma})_{\sigma \in \mathcal{E}})\;:\; v_{K} \in \RR,\, v_{\sigma} \in \RR,\\
& v_{\sigma}=0 \; \mbox{for all}\; \sigma \in \edgesext\mbox{ such that }\sigma\subset \Gamma_{1} \}
\end{aligned}
\end{equation}
and
\begin{equation}\label{def.XDG1}
\begin{aligned}
X_{\disc,\Gamma_1}=\{ &v=((v_{K})_{K\in \mathcal{M}}, (v_{\sigma})_{\sigma \in \mathcal{E}})\;:\; v_{K} \in \RR,\, v_{\sigma} \in \RR,\\
&v_K=0 \mbox{ for all } K\in \mesh,\,
v_\sigma=0 \mbox{ for all } \sigma \in \edgesint \mbox{ and }\\
&v_\sigma=0 \mbox{ for all } \sigma \in \edgesext\mbox{ such that }\sigma\subset \Gamma_2\cup\Gamma_3 \}.
\end{aligned}
\end{equation}
The space 
\begin{equation}\label{def:XDHMM}
X_{\disc}=\{v=((v_K)_{K\in\mesh},(v_\edge)_{\edge\in\edges})\,:\,v_K\in\RR\,,\;v_\edge\in\RR\}
\end{equation}
is the direct sum of these two spaces. The piecewise-constant function reconstruction $\Pi_\disc$, the piecewise-constant trace reconstruction $\mathbb T_\disc$ and the gradient reconstruction $\nabla_\disc$ are given by: $\forall v\in X_\disc$,
\begin{align}\label{gd-hmm-nlsign}
\begin{aligned}
&\forall K\in\mathcal M\,:\,\Pi_\disc v=v_K\mbox{ on $K$},\\
&\forall \sigma\in\mathcal E_{\rm ext}\,:\,\mathbb{T}_\disc v = v_\sigma\mbox{ on $\sigma$},\\
&\forall K\in\mathcal M,\,\forall \sigma\in\mathcal E_K\,:\,\nabla_\disc v=\nabla_{K}v+
\frac{\sqrt{d}}{d_{K,\sigma}}(A_{K}R_{K}(v))_{\sigma}\mathbf{n}_{K,\sigma}\mbox{ on }D_{K,\edge},
\end{aligned}
\end{align}
where $D_{K,\edge}$ is the convex hull of $\edge\cup\{x_K\}$ and
\begin{itemize}
\item $\nabla_{K}v= \frac{1}{|K|}\sum_{\sigma\in \edgescv}|\sigma|v_\edge\mathbf{n}_{K,\sigma}$,
\item $R_{K}(v)=(v_{\sigma}-v_{K}-\nabla_{K}v\cdot (\centeredge-x_{K}))_{\sigma\in\mathcal E_K}\in \RR^\edgescv$, \item $A_K$ is an isomorphism of the vector space ${\rm Im}(R_K)$.
\end{itemize}

The interpolant $\cI_{\disc,\Gamma_1}:W^{1-\frac{1}{p},p}(\partial\O)\to X_{\disc,\Gamma_1}$ is defined by
\begin{equation}\label{def:ID}
\begin{aligned}
\forall g\in W^{1-\frac{1}{p},p}(\partial\O)\,:\,{}&(\cI_{\disc,\Gamma_{1}}g)_{\sigma}=\frac{1}{|\sigma|}\int_\sigma g(x)\ud s(x),\\
&\mbox{ for all } \sigma \in\edgesext \mbox{ such that } \sigma\subset\Gamma_1.
\end{aligned}
\end{equation}

We then have $\mathcal{K}_{\disc}:= \{ v \in X_{\disc,\Gamma_{2,3}}+\cI_{\disc,\Gamma_1}g \; : \; v_{\sigma} \leq a
\mbox{ on }\sigma\,,\; \forall \sigma \in \edgesext \mbox{ s.t.\ } \sigma\subset \Gamma_{3}\}$,
and
the HMM discretisation of Problem \eqref{wenlsign} is the gradient scheme \eqref{gsnlsign} corresponding to the gradient discretisation described above.

The convergence of the HMM scheme is a consequence of Theorem \ref{thm:convnlsign} and of the four properties proved in the following proposition. The condition \eqref{newcons} follows from the GD-consistency of HMM for Fourier boundary conditions (see \cite[Definition 3.37 and Section 13.2.2]{S1}).

\begin{proposition}\label{pro-hmm-consist}
Let $(\disc_m)_{m\in\NN}$ be a sequence of HMM GDs given by \eqref{def.XDG}--\eqref{gd-hmm-nlsign},
for certain polytopal meshes $(\polyd_m)_{m\in\NN}$. Assume the existence of $\theta>0$ such that,
for any $m\in\NN$,
\begin{equation}\label{reg.HMM.1}
\begin{aligned}
\theta_\mesh:=\max_{K\in\mesh_m}&\left(\dsp\max_{\edge\in\edges_K}\frac{h_K}{d_{K,\edge}}+{\rm Card}(\edges_K)\right)
+\max_{\stackrel{\mbox{\scriptsize $\edge\in \edges_{m,\rm int}$}}{\mesh_\edge=\{K,L\}}}\left( \frac{d_{K,\edge}}{d_{L,\edge}}+\frac{d_{L,\edge}}{d_{K,\edge}} \right)\leq \theta
\end{aligned}\end{equation}
and, for all $K\in \mesh_m$ and $\mu \in \RR^{\edges_K}$,
\begin{equation}\label{reg.HMM.2}
\begin{aligned}
\frac{1}{\theta}\dsp\sum_{\sigma\in\edges_K}|D_{K,\sigma}| \left| \frac{R_{K,\sigma}(\mu)}{d_{K,\sigma}} \right|^p
\leq{}& \dsp\sum_{\sigma\in\edges_K} |D_{K,\sigma}| \left| \frac{(A_K R_K(\mu))_\sigma}{d_{K,\sigma}} \right|^p\\
\leq{}& \theta \dsp\sum_{\sigma\in\edges_K} |D_{K,\sigma}| \left| \frac{R_{K,\sigma}(\mu)}{d_{K,\sigma}} \right|^p.
\end{aligned}
\end{equation}
Then the sequence $(\disc_m)_{m\in\NN}$ is coercive, limit-conforming and compact in the sense of Definitions \ref{def:nlsigncoer}, \ref{def:nlsignlim} and \ref{def:comnlobs}. If moreover the function $a$ is piecewise-constant on $\edgesext$, then the sequence $(\disc_m)_{m\in\NN}$ is GD-consistent
in the sense of Definition \ref{def:nlsigncons}.
\end{proposition}

\begin{remark}[About the assumptions \eqref{reg.HMM.1}--\eqref{reg.HMM.2}]
Assumption \eqref{reg.HMM.1} is a regularity assumption on the sequences of meshes.
The boundedness of $h_K/d_{K,\edge}$ imposes, by \cite[Lemma B.1]{S1}, that each cell is star-shaped with respect to all points in a ball of radius comparable to the diameter of the cell. The presence of
${\rm Card}(\edges_K)$ forces the maximum number of faces of each cell to remain uniformly bounded.
Finally, the last addend in \eqref{reg.HMM.1} forces the `centers' of two cells $K$ and $L$ on each side of a given face $\edge$ to be within comparable distance to the face (combined with the first bound, this somehow stipulates that two neighbourhing cells must have comparable diameters).

On the contrary, Assumption \eqref{reg.HMM.2} deals with the `coefficients' of the chosen HMM scheme, since it imposes the uniform boudedness of the isomorphisms $A_K$ and their inverse in an appropriately scaled $l^p$ norm on each space ${\rm Im}(R_{K})$.
\end{remark}

\begin{remark} In \cite{A-2-19}, the convergence of numerical schemes for variational inequalities
(with homogenous Dirichlet BC and constant barriers) is established by using the density of
$C^2(\overline \O)\cap \cK$ in $\cK$. We do not need such a density result in Proposition
\ref{pro-hmm-consist}, which enables us to treat the case of piecewise-constant
barriers.
\end{remark}

\begin{proof}
The coercivity, limit-conformity and compactness follow as in the case of the HMM
method for PDEs, see \cite[Theorem 13.14 and Section 13.2]{S1}.
To prove the GD-consistency, we make use of the interpolation operator described in the
appendix. Let $\varphi\in \mathcal K$ and take $v_m=P^{\omega_m}_{\disc_m}\varphi$,
where $P^{\omega_m}_{\disc_m}$ is defined by \eqref{def:Ppolyd} with weights
$\omega_m=(\omega_{m,K})_{K\in\mesh_m}$ given by Lemma \ref{lem:exist.weights}.
Since $\gamma(\varphi)=g$ on $\Gamma_1$, the definitions \eqref{def:ID} and \eqref{def:Ppolyd} of $\cI_{\disc_m,\Gamma_1}$ 
and $P^{\omega_m}_{\disc_m}$ show that $v_m\in \cI_{\disc_m,\Gamma_1}g+X_{\disc_m,\Gamma_{2,3}}$.
Moreover, if $\edge\subset\Gamma_3$ then, since $a$ is constant on $\edge$,
\begin{equation}\label{vm:bdry}
(v_m)_{\edge}=\frac{1}{|\edge|}\int_\edge \gamma(\varphi)(x)\ud s(x)\le 
\frac{1}{|\edge|}\int_\edge a(x)\ud s(x)=a_{|\edge}.
\end{equation}
Hence, $v_m\in \mathcal K_{\disc_m}$ and thus
\[
S_{\disc_m}(\varphi)\le \|\Pi_{\disc_m} v_m-\varphi\|_{L^p(\O)}+\|\nabla_{\disc_m} v_m-\nabla\varphi\|_{L^p(\O)^d}.
\]
Proposition \ref{prop:interp.HMM} shows that the right-hand side of this inequality tends to $0$
as $m\to\infty$, which concludes the proof of the GD-consistency of $(\disc_m)_{m\in\NN}$.
\end{proof}

\begin{remark}[Non-piecewise-constant barrier]\label{rem:HMM.a.nc}
If the barrier $a$ is not piecewise-cons\-tant on $\edgesext$, in the context of HMM schemes
it is natural to consider an approximate barrier as in Section \ref{sec:approx.barriers}. 
The function $a_\disc$ is simply defined as the piecewise-constant function such that
\[
\forall \edge\in\edgesext\,,\;(a_\disc)_{|\edge}=\frac{1}{|\edge|}\int_\edge a(x)\ud s(x).
\]
The same reasoning as in Step 1 of the proof of Proposition \ref{prop:interp.HMM} shows that,
under the assumptions of Proposition \ref{pro-hmm-consist}, $a_{\disc_m}\to a$ in $L^p(\partial\O)$
as $m\to\infty$. Moreover, if $\varphi\in \mathcal K$ then the inequality in \eqref{vm:bdry} shows
that $v_m$ constructed in the proof above belongs to $\mathcal K_{\disc_m,a_{\disc_m}}$.
Hence, $(\disc_m)_{m\in\NN}$ remains GD-consistent if $S_{\disc_m}$ is defined using
$\mathcal K_{\disc_m,a_{\disc_m}}$ instead of $\mathcal K_{\disc_m}$.
Theorem \ref{thm:bari-conv-nlsign} can therefore be applied and establishes the convergence of
the HMM method with this approximate barrier.
\end{remark}

\begin{remark}[Non-conforming $\mathbb{P}_1$ finite elements]
With minor modifications in the proof of Proposition \ref{prop:interp.HMM} regarding
the approximation inside the cell, the arguments
above can be used to analyse the gradient discretisations corresponding to
non-conforming $\mathbb{P}_1$ finite elements \cite[Chapter 9]{S1}. This
shows that Theorems \ref{thm:convnlsign} and \ref{thm:bari-conv-nlsign} apply to these
non-conforming finite elements.
\end{remark}

\subsection{HMM methods for the obstacle problem and Bulkley model}
We use the notations introduced in Section \ref{sec:HMM-nlsign}. The elements of gradient discretisation $\disc$ to consider here are given by  
\begin{equation*}
\begin{aligned}
X_{\disc,0}=\{ v=((v_{K})_{K\in \mathcal{M}}, (v_{\sigma})_{\sigma \in \mathcal{E}})\;:\;{}& v_{K} \in \RR,\;  v_{\sigma} \in \RR,\; v_{\sigma}=0 \; \mbox{for all}\; \sigma \in \edgesext\},\\
X_{\disc,\dr\O}=\{ v=((v_{K})_{K\in \mathcal{M}}, (v_{\sigma})_{\sigma \in \mathcal{E}})\;:\;{}& v_{\sigma} \in \RR,\;  v_K=0 \mbox{ for all } K\in \mesh,\\
& v_\sigma=0 \mbox{ for all } \sigma \in \edgesint\}.
\end{aligned}
\end{equation*}
The discrete mappings $\cI_{\disc,\dr\O}$, $\Pi_\disc$ and $\nabla_\disc$ are as in Section \ref{sec:HMM-nlsign}.

\par Setting $ \mathcal{K}_{\disc}:= \{ v \in X_{\disc,0}+\cI_{\disc,\partial\O}h \; : \; v_K \leq \psi
\mbox{ on $K$, for all $K \in \mesh$}\}$, the HMM methods for \eqref{wenlobs} and \eqref{webulk} are respectively the gradient schemes \eqref{gsnlobs} and \eqref{gsbulk} comming from the above gradient discretisation.  

Recall that the coercivity, limit-conformity and compactness for sequences of GDs
adapted to the obstacle problem are the same properties as for sequences
of GDs for PDEs with non-homogeneous Dirichlet boundary conditions.
The proof of these properties follow therefore from \cite{S1}, under 
the regularity assumptions \eqref{reg.HMM.1}--\eqref{reg.HMM.2}.
If the barrier $\psi$ is constant in each cell,
the GD-consistency follows as in Proposition \ref{pro-hmm-consist}. Indeed, the weights $\omega_{m,K}$ being
non-negative, for all $\varphi\in\mathcal K$, $m\in\NN$ and $K\in\mesh_m$, we have
\begin{align*}
(P^{\omega_m}_{\disc_m}\varphi)_{|K}={}&\frac{1}{|K|}\int_K \omega_{m,K}(x)\varphi(x)\ud x\\
\le{}& \frac{1}{|K|}\int_K \omega_{m,K}(x)\psi(x)\ud x\\
={}&\psi_{|K}\frac{1}{|K|}\int_K \omega_{m,K}(x)\ud x=\psi_{|K},
\end{align*}
which shows that $P^{\omega_m}_{\disc_m}\varphi\in \mathcal K_{\disc_m}$.

For the Bulkley model, all the properties of GDs are identical to those for PDEs with
homogeneous Dirichlet boundary conditions, and therefore follow (still under the
assumptions \eqref{reg.HMM.1}--\eqref{reg.HMM.2}) from \cite{B1}.

Using these properties, the convergence of the HMM method for each problem is a straightforward consequence of Theorems \ref{thm:convnlobs} and \ref{thm:convbulkley}.   

\begin{remark}[Non-piecewise-constant obstacle]
If $\psi$ is not piecewise-constant on the mesh,
we approximate it by the piecewise-constant function $\psi_\disc$ defined by
\[
\forall K\in\mesh\,,\;(\psi_\disc)_{|K}=\frac{1}{|K|}\int_K \omega_K(x)\psi(x)\ud x.
\]
Step 1 in the proof of Proposition \ref{prop:interp.HMM} shows that
$\psi_{\disc_m}\to\psi$ in $L^p(\Omega)$
as $m\to\infty$. Since $P^{\omega_m}_{\disc_m}\varphi\in \mathcal K_{\disc_m,\psi_{\disc_m}}$
whenever $\varphi\in \mathcal K$, this establishes the GD-consistency and
Theorem \ref{thm:bari-conv-nlobs} then ensures the convergence of the HMM method
for the non-linear obstacle problem with approximate barriers.
\end{remark}



\section{Numerical results}\label{sec:test}
\par We demonstrate here the efficiency of the HMM method for solving non-linear Signorini problems by considering the meaningful example of the seepage model. Due to the double
non-linearity in the model, two iterative algorithms are used in conjunction to compute a numerical solution: fixed point iterations to deal with the non-linear operator, and a monotonicity algorithm for the inequalities coming from the imposed Signorini boundary conditions.

We consider test cases from \cite{A1}. In each case, letting $x=(x_1,x_2)$, the model reads
\begin{align*}
-\div(\Lambda(x,\bar u)\nabla\bar u)= 0 &\mbox{\quad in $\Omega$,}\\
\bar{u} =g &\mbox{\quad on $\Gamma_{1}$,} \label{nlsign2}\\
\Lambda(x,\bar u)\nabla\bar u\cdot \mathbf{n}= 0 &\mbox{\quad on $\Gamma_{2}$,}\\
\left.
\begin{array}{r}
\dsp \bar{u} \leq x_2\\
\dsp \Lambda(x,\bar u)\nabla\bar u\cdot \mathbf{n} \leq 0\\
\dsp \Lambda(x,\bar u)\nabla\bar u\cdot \mathbf{n} (x_2-\bar{u}) = 0
\end{array} \right\} 
& \mbox{\quad on}\; \Gamma_{3},
\end{align*}
with $\Omega$, $\Gamma_1$, $\Gamma_2$ and $\Gamma_3$ depending on the test case.
The medium is considered homogeneous and anisotropic and so, after scaling, we can assume that the permeability tensor $\mathbf{K}$ is the identity. Following Remark \ref{rem:HK}, we therefore set $\Lambda(x,s)=\cH_{\varepsilon}^\lambda(s-x_2)$ 
with a regularised Heavyside function $\cH_{\varepsilon}^\lambda$ given by
\begin{equation}\label{def:Hreg}
\begin{aligned}
\mathcal H_\varepsilon^\lambda(\rho)= \left\{
\begin{array}{ll}
1 & \mbox{if } \rho\geq 0\,,\\
\frac{1-\varepsilon}{\lambda}\rho+1 & \mbox{if } -\lambda<\rho<0\,,\\
\varepsilon& \mbox{if } \rho<-\lambda.
\end{array} \right.
\end{aligned}
\end{equation}
Here, both $\lambda$ and $\varepsilon$ are taken equal to $10^{-3}$. As stated above, to obtain the solution to this problem, we first apply simple fixed point iterations (Algorithm \ref{algorm-fixed}), the idea of which is to generate a sequence $(u^{(n)})_{n\in\NN}\subset \cK_\disc$ by solving linear variational inequalities. These linear VIs are obtained by
fixing the non-linearity in the operator to the previous element in the sequence.

\begin{algorithm}
\caption{Fixed point algorithm}\label{algorm-fixed}
\begin{algorithmic}[1]
\State Let $\delta$ be a small number (stopping criteria) and $u^{(0)}=0$ \Comment{For us, $\delta=10^{-2}$.}
\For {$n=1,2,3,...$}
\State Solve the following linear VI, using Algorithm \ref{algo:monoton}: \Comment{$u^{(n)}$ is known} 
\begin{equation}\label{eq:nlalgorithm}
\left\{
\begin{array}{l}
\dsp \mbox{Find}\; u^{(n+1)} \in \mathcal K_\disc \mbox{ such that, for all } v\in \mathcal{K}_{\disc}\,,\\
\dsp\int_\O\Lambda(x,\Pi_\disc u^{(n)})\nabla_\disc u^{(n+1)}\cdot \nabla_\disc(u^{(n+1)}-v) \ud x\\
\qquad\leq 
\dsp\int_\O f\Pi_\disc(u^{(n+1)}-v)\ud x.
\end{array}
\right.
\end{equation}
\If {
\begin{multline*}
\| \Pi_\disc(u^{(n+1)}-u^{(n)}) \|_{L^2(\O)}+\| \nabla_\disc(u^{(n+1)}-u^{(n)}) \|_{L^2(\O)^d} \\
\leq \delta( \| \Pi_\disc u^{(n)} \|_{L^2(\O)}+\| \nabla_\disc u^{(n)} \|_{L^2(\O)^d})
\end{multline*}} 
\State Exit ``for'' loop
\EndIf
\EndFor
\State Set $u=u^{(n+1)}$
\end{algorithmic}
\end{algorithm}
In each iteration $n$ in Algorithm \ref{algorm-fixed}, a linear VI must be solved. To compute its solution, introduce the linear fluxes $u\mapsto F^w_{K,\sigma}(u)$ (for $K\in\mesh$ and $\sigma\in\edges_K$) defined by:
for all $K \in \mesh$ and all $u,v,w \in X_\disc$,
\begin{align*}
\sum_{\sigma \in \mathcal{E}_K}|\sigma| F_{K,\sigma}^{w}(u)
(v_K-v_\sigma)={}&\int_K \Lambda(x,w_K)\nabla_\disc u\cdot\nabla_\disc v \ud x
\end{align*}    

Choosing $w=u^{(n)}$ in this relation, Problem \eqref{eq:nlalgorithm} can be recast as \cite{AD14}
\begin{align}
\sum_{\sigma \in \mathcal{E}_K}|\sigma|F_{K,\sigma}^{u^{(n)}}(u^{(n+1)})= |K|f_K, &\quad \forall K \in \mesh \label{hmm-nlsig1}\\
F_{K,\sigma}^{u^{(n)}}(u^{(n+1)})+F_{L,\sigma}^{u^{(n)}}(u^{(n+1)})= 0, &\quad \forall \sigma\in\mathcal E_{\rm int}\mbox{ with }
\mesh_\sigma=\{K,L\},\label{hmm-nlsig2}\\
u_\sigma^{(n+1)} = g, &\quad \forall \sigma \in \mathcal{E}_{\rm ext} \mbox{ such that }\sigma\subset  \Gamma_1,\label{hmm-nlsig3}\\
F_{K,\sigma}^{u^{(n)}}(u^{(n+1)})=0, &\quad \forall K \in \mesh\,,\forall \sigma \in \mathcal{E}_K \mbox{ such that }\sigma\subset  \Gamma_2,\label{hmm-nlsig4}\\
F_{K,\sigma}^{u^{(n)}}(u^{(n+1)})(u_\sigma^{(n+1)} - \overline x_{2,\sigma})=0 , & \quad\forall K \in \mesh\,, \forall \sigma \in \mathcal{E}_K \mbox{ such that }\sigma\subset  \Gamma_3,\label{hmm-nlsig5}\\
-F_{K,\sigma}^{u^{(n)}}(u^{(n+1)})\leq 0, &\quad \forall K \in \mesh\,, \forall \sigma \in \mathcal{E}_K \mbox{ such that }\sigma\subset  \Gamma_3,\label{hmm-nlsig6}\\
u_\sigma^{(n+1)} \leq \overline x_{2,\sigma},&\quad \forall \sigma \in \mathcal{E}_{\rm ext} \mbox{ such that }\sigma\subset \Gamma_{3}\label{hmm-nlsig7}.
\end{align}
Here $\overline x_{2,\sigma}$ denotes the second coordinate of the centre of mass of edge $\sigma$.
This choice corresponds to the approximate barrier $a_\disc$ of $a(x)=x_2$ described
in Remark \ref{rem:HMM.a.nc}. The monotonicity algorithm given in \cite{A2} is used to solve this non-linear system (see Algorithm \ref{algo:monoton}). It is proved in \cite{A2-23} that the number of
iterations of this monotonicity algorithm is bounded by the number of edges in $\Gamma_3$. 
This algorithm only requires, at each of its steps, to solve a square linear system on the
unknowns $((w_K)_{K\in\mesh},(w_\edge)_{\edge\in\edges})$. 
\begin{algorithm}
\caption{Monotonicity algorithm}\label{algo:monoton}
\begin{algorithmic}[1]
\State (Only the first time the algorithm is called):

\hspace*{-2em}\begin{minipage}{0.4\linewidth}
Set $A^{(0)}=\{ \edge \in \edges \; : \; \edge \subset \Gamma_3 \}$, $B^{(0)}=\emptyset$
and $I={\rm Card}(A^{(0)})$
\end{minipage} \Comment{$I$= theoretical bound on the iterations}
\While{$i\le I$} 
\State $A^{(i)}$ and $B^{(i)}$ being known, find the solution $w$ to the system
\eqref{hmm-nlsig1}--\eqref{hmm-nlsig4} together with
\begin{equation}\label{eq:monotone-flux}
\begin{aligned}
\-F_{K,\sigma}^{u^{(n)}}(w)= 0, &\quad \forall K \in \mesh\,, \forall \sigma \in \mathcal{E}_K \mbox{ such that }\sigma\in  B^{(i)},\\
w_\sigma=\overline x_{2,\sigma},&\quad \forall \sigma \in \mathcal{E}_{\rm ext} \mbox{ such that }\sigma \in A^{(i)}.
\end{aligned}
\end{equation}
\State Set $A^{(i+1)}=\{ \sigma\in A^{(i)}\; :\; -F_{K,\edge}^{u^{(n)}}(w)\leq 0 \} \cup \{ \sigma\in B^{(i)}\; :\; w_\sigma\geq \overline x_{2,\sigma} \} $
\State Set $B^{(i+1)}=\{ \sigma\in B^{(i)}\; :\; w_\edge<\overline{x}_{2,\edge} \} \cup \{ \sigma\in A^{(i)}\; :\; -F_{K,\sigma}^{u^{(n)}}(w)>0 \} $
\If {$A^{(i+1)}=A^{(i)}$ and $B^{(i+1)}=B^{(i)}$} 
\State {Exit ``while'' loop}
\EndIf
\EndWhile
\State Set $u^{(n+1)}=w$ \Comment{Solution to \eqref{hmm-nlsig1}--\eqref{hmm-nlsig7}}
\State (For next call of Algorithm \ref{algo:monoton}) Set $A^{(0)}=A^{(i+1)}$ and $B^{(0)}=B^{(i+1)}$
\end{algorithmic}
\end{algorithm}

\subsection{Test 1}\label{test1}

The geometry of the domain $\O$ representing the dam is illustrated in Fig \ref{fig:geomdam}. We have here
\begin{align*}
\Gamma_1={}&\{ (x_1,x_2)\in \RR^2:\, x_1=0 \mbox{ and } x_2\in[0,5] \} \\
&\cup \{ (x_1,x_2)\in \RR^2:\, x_1+x_2=7 \mbox{ and } x_2\in [0,1]\},\\
\Gamma_2={}&\{ (x_1,x_2)\in\RR^2:\, x_2=0 \},\\
\Gamma_3={}&\{ (x_1,x_2)\in \RR^2:\, x_2=5 \mbox{ and } x_1\in[0,2] \} \\
&\cup \{ (x_1,x_2)\in \RR^2:\, x_1+x_2=7 \mbox{ and } x_2\in (1,5]\}.
\end{align*}
\begin{center}
\begin{figure}[!h]
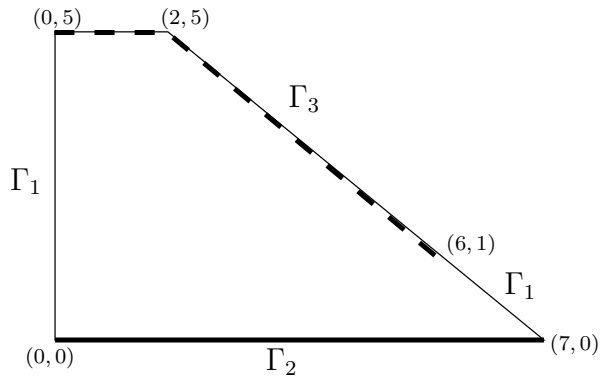
\caption{Test 1, geometry of the dam.}
\label{fig:geomdam}
\end{figure}
\end{center}
The boundary conditions $g$ is defined by $g(0,x_2)=5$ for all $x_2\in [0,5]$ and $g(x_1,x_2)=1$, for
all $x_1\in (0,7)$.

The test is conducted on two different families of meshes shown in Fig.\ \ref{fig:mesh-domain-seepage}. The first family is (mostly) made of hexahedral meshes: \texttt{Hexa1} has 441 cells, $h_\mesh\approx 0.69$ and a number of edges in $\Gamma_3$ equal to $N=72$; \texttt{Hexa2} has 1681 cells, $h_\mesh\approx 0.36$ and $N=144$. The second family of meshes are `Kershaw' meshes from \cite{HH08}: \texttt{Kershaw1} has 2601 cells, $h_\mesh\approx 0.69$ and $N=92$; \texttt{Kershaw2} has 4626 cells, $h_\mesh\approx 0.52$ and $N=122$.

For the two hexagonal meshes, the fixed point algorithm (Algorithm \ref{algorm-fixed}) converges in 5 and 4 iterations, respectively. For \texttt{Kershaw1}, an oscillating
phenomenon occurs: for $n\ge 9$, $u^{(n)}\approx u^{(n-2)}$ but $u^{(n)}\not= u^{(n+1)}$; the Ka\c{c}anov algorithm does not converge, but essentially alternates between two vectors. To break this oscillation, we use an under-relaxation technique: when it is found that $|u^{(n)}-u^{(n-2)}|_\infty\le 10^{-2}|u^{(n-2)}|_\infty$, where $|\cdot|_\infty$ is the maximum norm of vectors in $X_{\disc}$, denoting by $\widetilde{u}^{(n+1)}$ the solution of \eqref{eq:nlalgorithm} we actually set $u^{(n+1)}=u^{(n)}+0.5(\widetilde{u}^{(n+1)}-u^{(n)})$; that is, we only progress halfway from $u^{(n)}$ to $\widetilde{u}^{(n+1)}$. This tweak enables the fixed-point algorithm
to converge in 11 iterations overall, which remains quite low given the distortion of the grid. We notice that this oscillation does not occur on \texttt{Kershaw2}, and that the fixed point algorithm converges on this mesh in 6 iterations, without the need for under-relaxation.

For all meshes, the maximum number of iterations of the monotonicity algorithm (Algorithm \ref{algo:monoton}) is also very far from the theoretical bound, with
$6$ and $9$ for the hexagonal meshes, and $5$ and $6$ for the Kershaw meshes. Note that, between two iterations of the fixed-point
algorithm, the sets $A$ and $B$ in the monotonicity algorithm are not reset. This means that the final sets obtained
at level $n$ of Algorithm \ref{algorm-fixed} are used as initial guesses at level $n+1$ of this algorithm.
This considerably reduces the number of iterations of Algorithm \ref{algo:monoton} and, after the first
2 or 3 iterations of Algorithm \ref{algorm-fixed},
the monotonicity algorithm converges in only 1 or 2 iterations.

\medskip

The monotonicity algorithm offers a way to determine the location of the seepage point. Following the interpretation of the model in \cite{A1}, the seepage point should split the free boundary $\Gamma_3$ into upper and lower parts in
the following way: (1) there is no flow on the upper part (so $F_{K,\sigma}=0$ for every edge $\sigma$ in this part);
(2) the pore pressure vanishes on the lower part (so $\bar u=x_2$ on this part); (3) both conditions are satisfied at the seepage point. The first and second conditions are naturally expressed by Equation \eqref{eq:monotone-flux}. Since any edge in the set $B$ cannot satisfy the last property (due to the strict inequality $w_\sigma < \overline{x}_{2,\sigma}$), the
seepage point does not lie on those edges. This point can thus be located at the edge $\sigma$ in the set $A$
whose midpoint has the largest ordinate $\overline{x}_{2,\sigma}$. Considering the mesh size and the fact that the HMM solution is computed at the mid-point of edges, our numerical results locate the seepage point
at an $x_2$-coordinate in $[3.31,3.65]$ for \texttt{Hexa1}, and in $[3.28,3.63]$ for \texttt{Kershaw1}. 
The seepage position moves only by $1\%$ from \texttt{Hexa1} to \texttt{Hexa2}, and by $2\%$ from \texttt{Kershaw1} to \texttt{Kershaw2}. This location is in perfect agreement with the numerical tests in \cite{A1}.

\begin{figure}[ht]
	\begin{center}
	\begin{tabular}{cc}
	\includegraphics[width=0.45\linewidth]{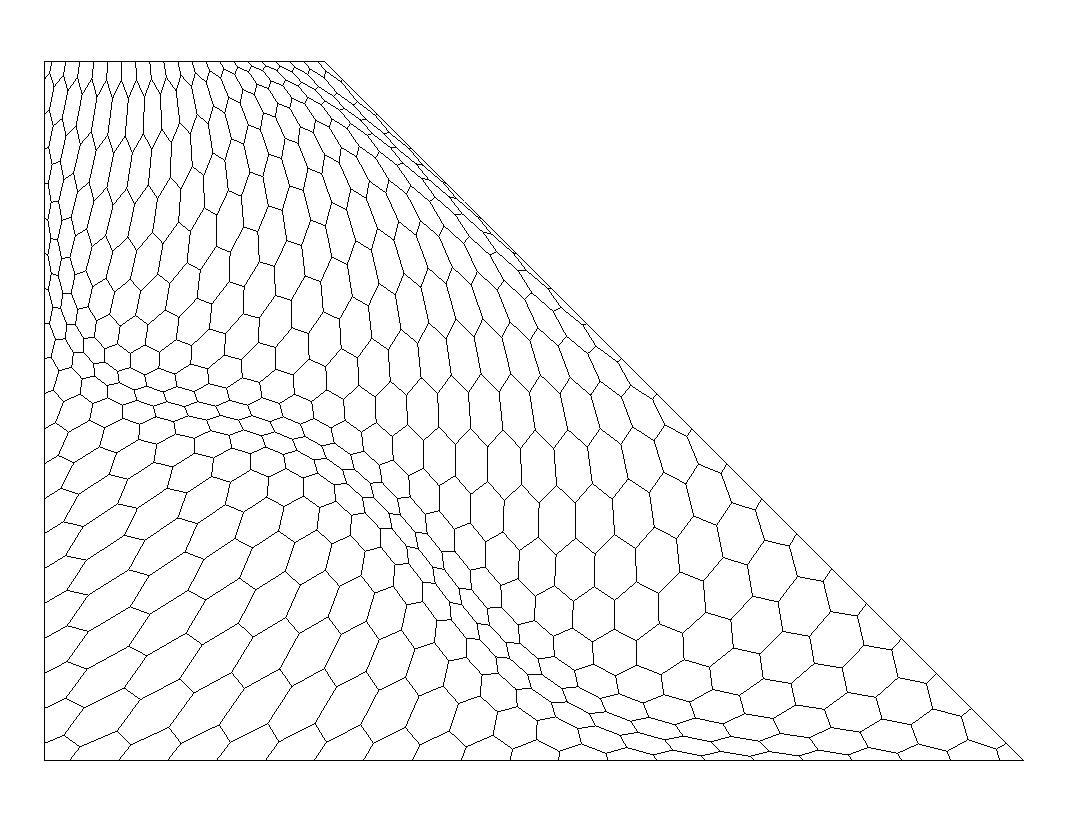} & \includegraphics[width=0.45\linewidth]{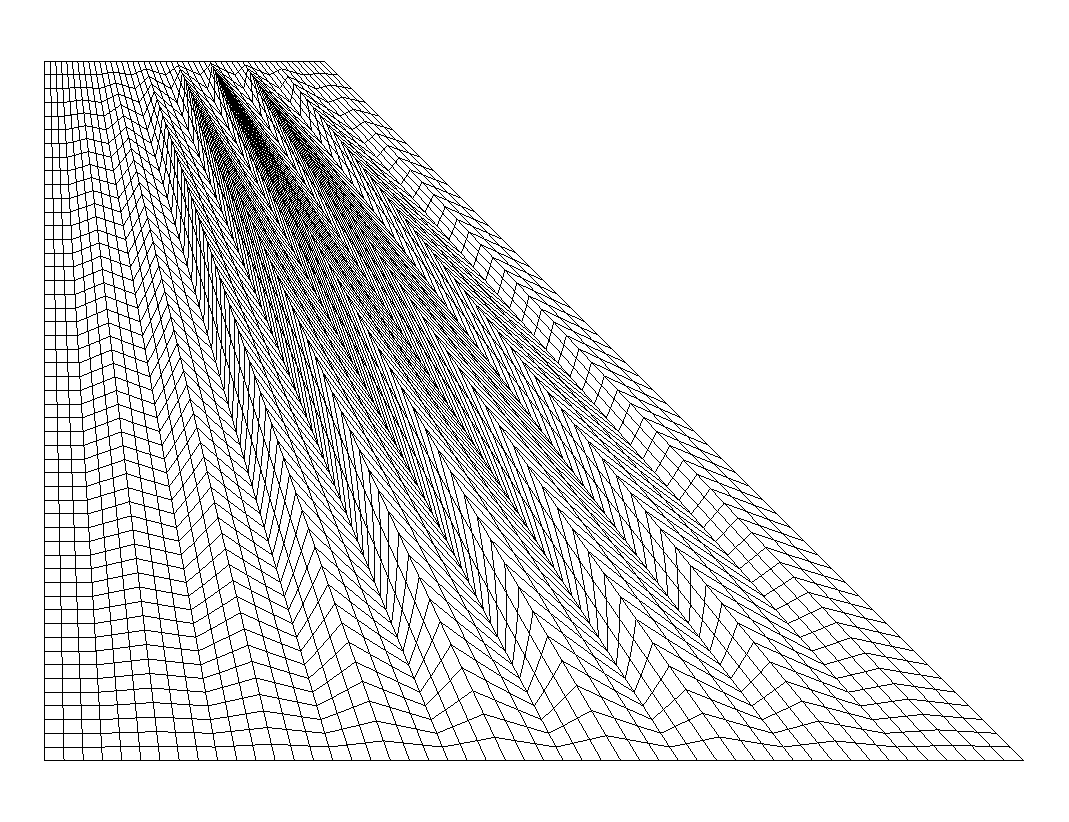}\\
	\texttt{Hexa1} (441 cells) & \texttt{Kershaw1} (2601 cells)\\
	\includegraphics[width=0.45\linewidth]{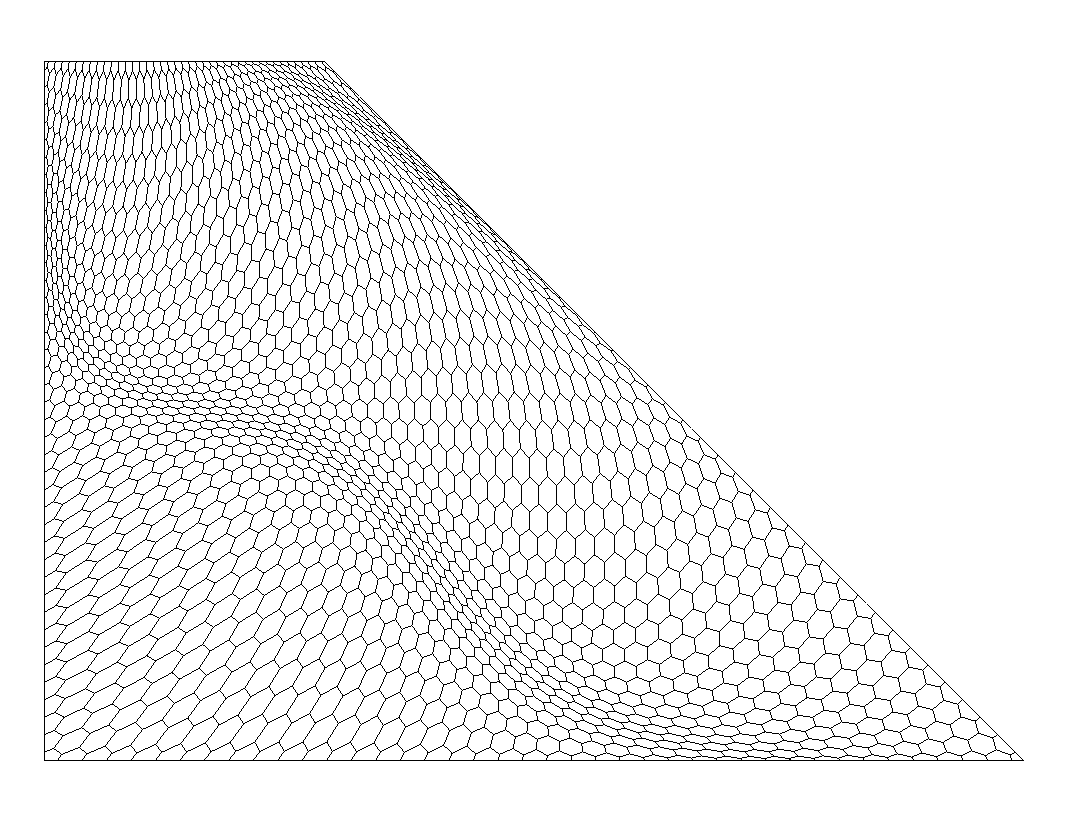} & \includegraphics[width=0.45\linewidth]{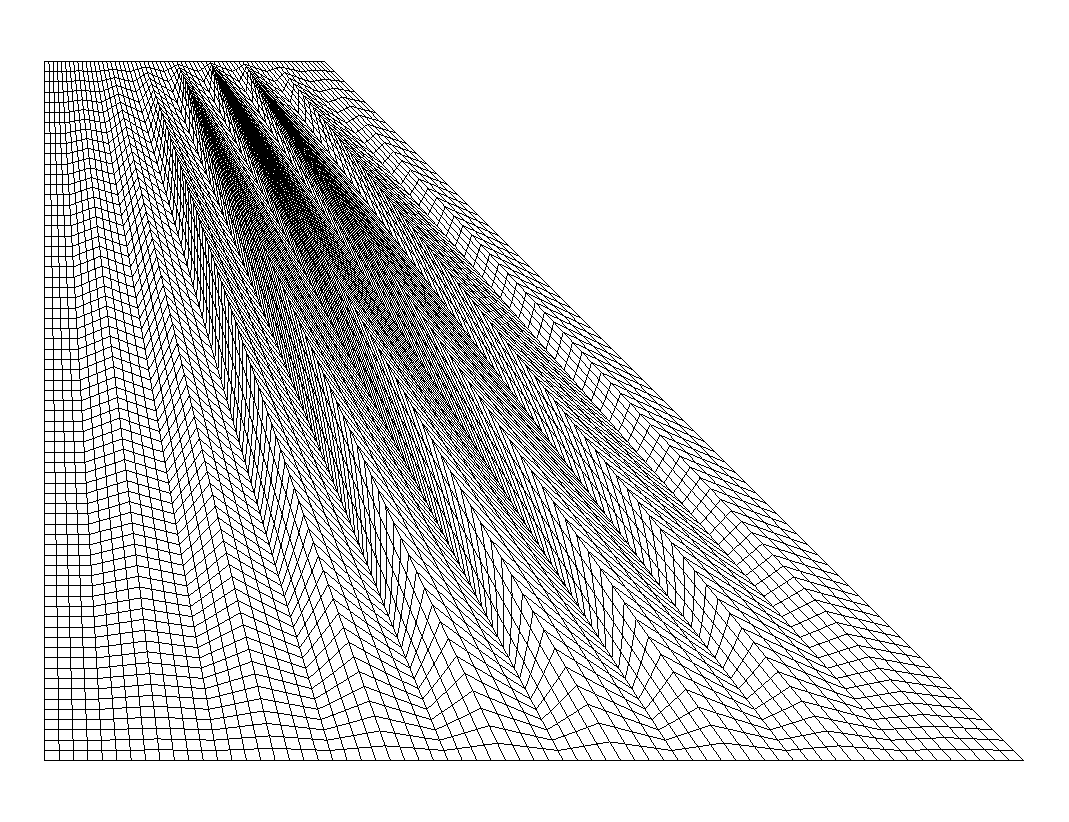}\\
	\texttt{Hexa2} (1681 cells) & \texttt{Kershaw2} (4624 cells)\\
	\end{tabular}
	\end{center}
	\caption{Test 1, elements of the hexahedral and Kershaw families of meshes.}
	\label{fig:mesh-domain-seepage}
\end{figure}

\begin{figure}[ht]
	\begin{center}
	\begin{tabular}{cc}
	\includegraphics[width=.5\linewidth]{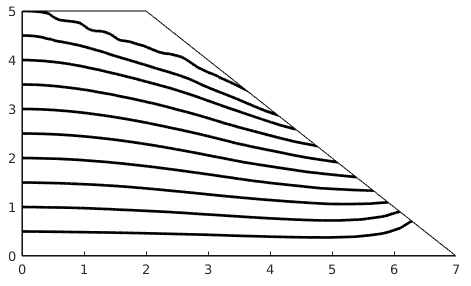} & \includegraphics[width=.5\linewidth]{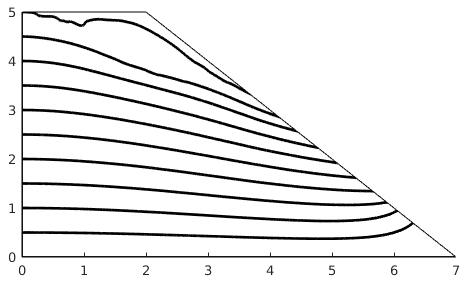}\\
	\texttt{Hexa1} (441 cells) & \texttt{Kershaw1} (2601 cells)\\
	\includegraphics[width=.5\linewidth]{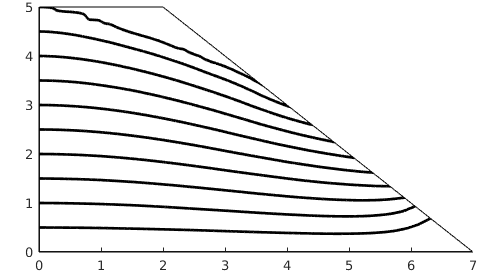} & \includegraphics[width=.5\linewidth]{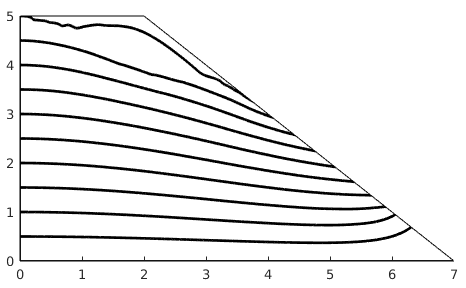}\\
	\texttt{Hexa2} (1681 cells) & \texttt{Kershaw2} (4624 cells)\\
	\end{tabular}
	\end{center}
	\caption{Test 1, streamlines on the hexahedral and Kershaw meshes.}
	\label{fig:streamlines-seepage}
\end{figure}


Fig \ref{fig:streamlines-seepage} shows the streamlines of the Darcy velocity field of the solution.
As expected, the distorted cells at the top of the domain provoke
perturbations of the streamlines there. Quite remarkably, though, this grid distortion does not impact the location of the seepage point. Elsewhere, the streamlines are very similar to the ones in \cite{A1}.
We notice that the streamlines are not extremely impacted by the mesh refinements, and remain rather distorted at the top of the domain for the Kershaw meshes. This is probably due to the fact that the regularity factor $\theta_\mesh$ defined in \eqref{reg.HMM.1} remains quite large for all these meshes ($\approx 164$ for \texttt{Kershaw1}, $\approx 173$ for \texttt{Kershaw2}), with maximum reached on cells where the top streamline flows. As an element of comparison, for hexahedral meshes the local regularity factors of the cells around this streamline is about $6$.

\subsection{Test 2}\label{test2}
The second test is on the seepage model with geometry of $\O$ describing a homogeneous isotropic embankment dam with horizontal under drain as shown in Fig.\ \ref{fig:geomdam2}. The boundary is split into
\begin{align*}
\Gamma_1={}&\{ (x_1,x_2)\in \RR^2:\, 29.81 x_2-20 x_1=0 \mbox{ and } x_2\in[0,18] \},\\
\Gamma_2={}&\{ (x_1,x_2)\in\RR^2:\, x_2=0 \mbox{ and } x_1 \in [0,59.62] \},\\
\Gamma_3={}&\{ (x_1,x_2)\in \RR^2:\, 29.81 x_2-20 x_1=0 \mbox{ and } x_2\in(18,20] \} \\
&\cup \{ (x_1,x_2)\in \RR^2:\, x_1\in [29.81,38.75]\mbox{ and } x_2=20\}\\
&\cup \{ (x_1,x_2)\in \RR^2:\, 29.81 x_2+20 x_1=1371.20 \mbox{ and } x_2\in [0,20]\}\\
&\cup \{ (x_1,x_2)\in \RR^2:\, x_2=0 \mbox{ and } x_1\in (59.62, 68.56]\}.
\end{align*}
The boundary data $g$ on $\Gamma_1$ is constant equal to $18$.
\begin{center}
\begin{figure}[!h]
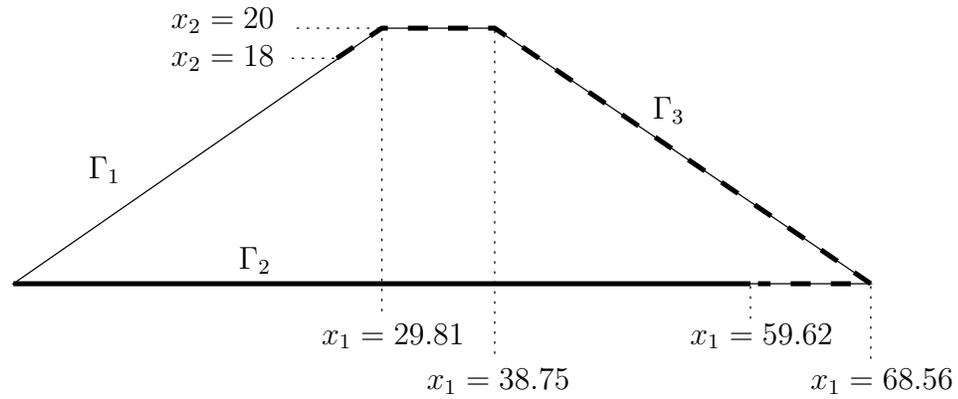
\caption{Test 2, sketch of geometry of the dam.}
\label{fig:geomdam2}
\end{figure}
\end{center}
\begin{figure}[ht]
	\begin{center}
	\begin{tabular}{cc}
	\includegraphics[width=1\linewidth]{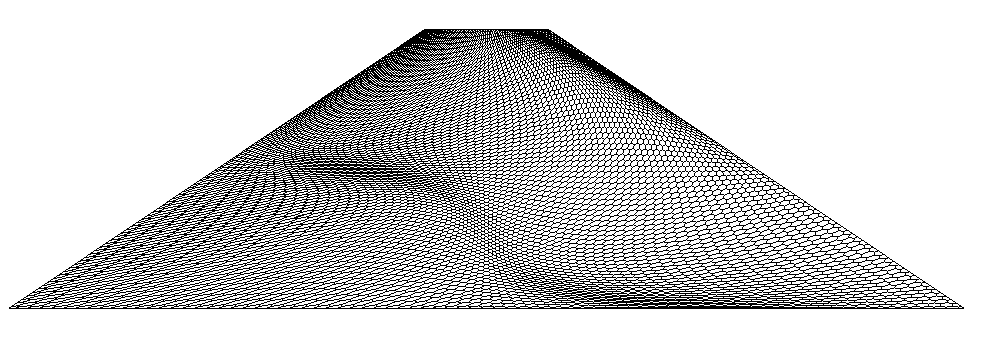} \\
	\texttt{Hexa3} (6561 cells)\\[.8em]
	\includegraphics[width=1\linewidth]{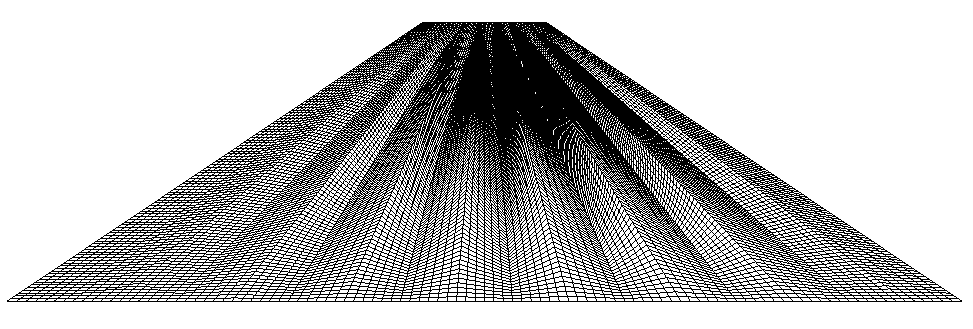} \\
	\texttt{Kershaw3} (10404 cells)\\
	\end{tabular}
	\end{center}
	\caption{Test 2, hexahedral and Kershaw meshes.}
	\label{fig:mesh-domain-seepage-test2}
\end{figure}

\begin{figure}[ht]
	\begin{center}
	\begin{tabular}{cc}
	\includegraphics[width=1\linewidth]{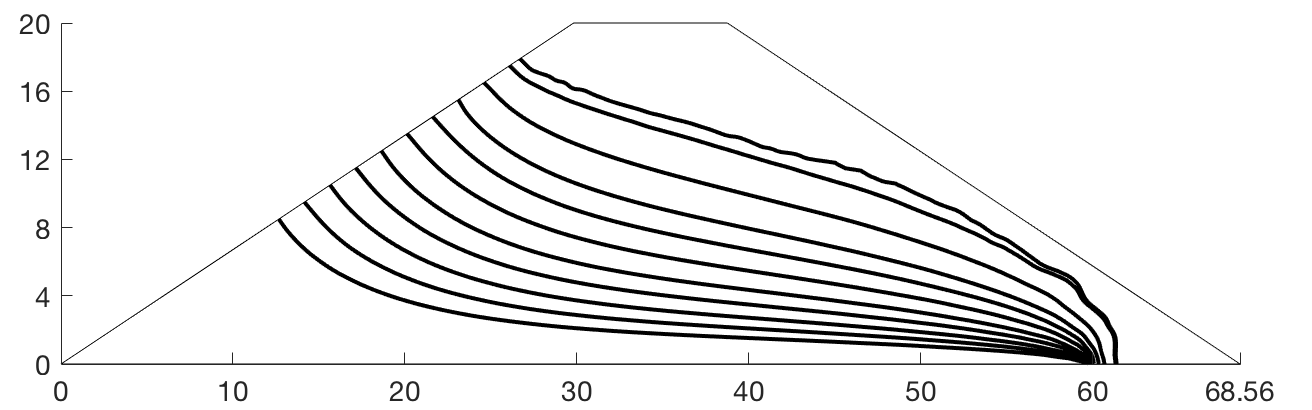} \\
	\texttt{Hexa3} (6561 cells)\\
	\includegraphics[width=1\linewidth]{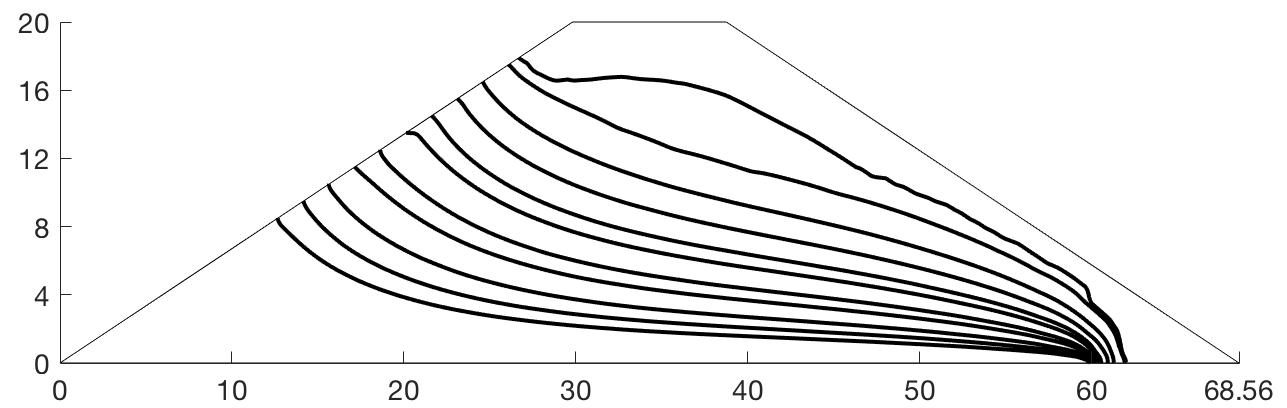} \\
	\texttt{Kershaw3} (10404 cells)\\
	\end{tabular}
	\end{center}
	\caption{Test 2, streamlines on the hexahedral and Kershaw meshes.}
	\label{fig:streamlines-test2}
\end{figure}

The test is performed on two different meshes, shown in Fig.\ \ref{fig:mesh-domain-seepage-test2}.
The first mesh (\texttt{Hexa3}) is (mostly) build on hexagonal cells with 6561 cells, maximum size $h_\mesh\approx 2.06$, and a number of edges in $\Gamma_3$ equal to $N=306$. The second one (\texttt{Kershaw3}) is a ``Kershaw mesh'' with 10404 cells, a maximum size $h_\mesh\approx 1.50$, and $N=181$. Using the under-relaxation technique detailed in Section \ref{test1}, the fixed point algorithm converges in 15 iterations for \texttt{Hexa3} and in 18 iterations for \texttt{Kershaw3}.
For both meshes, the maximum number of iterations of the monotonicity algorithm is 4, which is, as in Test 1, considerably lower than the number of edges in $\Gamma_3$.

\medskip

The streamlines of the Darcy velocity field are plotted in Fig.\ \ref{fig:streamlines-test2}. Despite perturbations of the streamlines at the top of the dam caused by distorted cells there, the streamlines behave mostly well and are similar to the ones presented in \cite{A1}. The numerical results indicate that the seepage position, corresponding to the point where the top streamlines meet the bottom of the dam, is approximately at $(61.6,0)$ for \texttt{Hexa3} and $(61.9,0)$ for \texttt{Kershaw3}.


\section{Appendix: interpolant for the HMM method}\label{sec:appen}

Let $\polyd$ be a polytopal mesh of $\O$, and
select weights $\omega=(\omega_K)_{K\in\mesh}$ with, for all $K\in\mesh$, $\omega_K\in L^\infty(K)$ such that
\begin{equation}\label{cond.weights}
\frac{1}{|K|}\int_K \omega_K(x)\ud x=1\quad\mbox{ and }\quad\frac{1}{|K|}\int_K x\omega_K(x)\ud x=x_K.
\end{equation}
Recalling the definition \eqref{def:XDHMM} of the space $X_\disc$,
the interpolant $P^\omega_\disc:W^{1,p}(\O)\to X_\disc$ is then defined by:
\begin{equation}\label{def:Ppolyd}
\begin{aligned}
&\forall \varphi\in W^{1,p}(\O)\,,\;P^\omega_\disc\varphi=((\varphi^\omega_K)_{K\in\mesh},(\varphi_\edge)_{\edge\in\edges})
\mbox{ with }\\
&\qquad\forall K\in\mesh\,,\; \varphi^\omega_K=\frac{1}{|K|}\int_K \omega_K(x)\varphi(x)\ud x\,,\\
&\qquad\forall \edge\in\edges\,,\;\varphi_\edge=\frac{1}{|\edge|}\int_\edge \varphi(x)\ud s(x).
\end{aligned}\end{equation}

This interpolant enjoys nice approximation properties for the HMM gradient discretisation. Before stating and proving these, we first establish the existence of weights with suitable properties. Note that
\cite[Lemma A.7]{DN17} already gives a construction of such weights (as linear functions),
without the positivity property.

\begin{lemma}[Existence of weights]\label{lem:exist.weights}
Let $\polyd$ be a polytopal mesh and let 
\[
\varrho_\polyd = \max_{K\in\mesh}\max_{\edge\in\edgescv}\frac{h_K}{d_{K,\edge}}.
\]
Then there exists weights $\omega=(\omega_K)_{K\in\mesh}$
satisfying \eqref{cond.weights} and such that
\begin{equation}\label{bound.weights}
\forall K\in\mesh\,,\;\forall x\in K\,,\; 0\le \omega_K(x)\le \varrho_\polyd^{d}.
\end{equation}
\end{lemma}

\begin{proof}
By \cite[Lemma B.1]{S1}, for all $K\in\mesh$ the ball $B_K$ of center $x_K$ and radius $\varrho_\polyd^{-1}h_K$
is fully contained in $K$. Let us define $\omega=(\omega_K)_{K\in\mesh}$ by
\[
 \forall K\in\mesh\,,\;\forall x\in K\,,\; \omega_K(x)=\left\{\begin{array}{ll} \frac{|K|}{|B_K|}&\mbox{ if $x\in B_K$},\\
0&\mbox{ if $x\not\in B_K$}.\end{array}\right.
\]
Denoting by $V_1$ the volume of the unit ball in $\RR^d$, since $K$ is contained in the ball of center $x_K$ and
radius $h_K$ we have $|K|\le V_1 h_K^d$. On the other hand, $|B_K|=V_1 (\varrho_{\polyd}^{-1}h_K)^d$.
Hence, $|K|/|B_K|\le \varrho_{\polyd}^d$ and \eqref{bound.weights} is satisfied.

The relations \eqref{cond.weights} are trivial since $\int_{B_K}1\ud x=|B_K|$ and $\int_{B_K} x\ud x=x_K|B_K|$
(as $x_K$ is the center of $B_K$).
\end{proof}

\begin{proposition}[Approximation properties of $P^\omega_\disc$]\label{prop:interp.HMM}
Let $(\polyd_m)_{m\in\NN}$ be a sequence of polytopal meshes such that $h_{\mesh_m}\to 0$ as $m\to\infty$
and, for some $\theta>0$, \eqref{reg.HMM.1} holds for all $m\in\NN$.
For each $m\in\NN$, take weights $\omega_m=(\omega_{m,K})_{K\in\mesh_m}$ given by Lemma \ref{lem:exist.weights}.

For $m\in\NN$, let $\disc_m=(X_{\disc_m},\Pi_{\disc_m},\mathbb{T}_{\disc_m},\nabla_{\disc_m})$ be the HMM gradient discretisations defined on $\polyd_m$ by \eqref{gd-hmm-nlsign}, without specific boundary
conditions.
Then, for all $\varphi\in W^{1,p}(\O)$, as $m\to\infty$,
\begin{align*}
&\Pi_{\disc_m}(P^{\omega_m}_{\disc_m}\varphi)\to\varphi\mbox{ in }L^p(\O),\\
&\mathbb{T}_{\disc_m}(P^{\omega_m}_{\disc_m}\varphi)\to \gamma(\varphi)\mbox{ in }L^p(\partial\O),\\
&\nabla_{\disc_m}(P^{\omega_m}_{\disc_m}\varphi)\to\nabla\varphi\mbox{ in }{L^p(\O)^d}.
\end{align*}
\end{proposition}

\begin{proof}

Note that by choice of $\theta$ and by Lemma \ref{lem:exist.weights}, for all $m\in\NN$
and $K\in \mesh_m$, $\|\omega_{m,K}\|_{L^\infty(K)}\le \theta^{d}$.

\medskip

\textbf{Step 1}: convergence of the function and trace reconstructions.

Fix $\varepsilon>0$ and take $\varphi_\varepsilon\in C^\infty_c(\RR^d)$ such that
$\|\varphi-\varphi_\varepsilon\|_{W^{1,p}(\O)}\le \varepsilon$. A triangle inequality yields
\begin{align}
\|\Pi_{\disc_m}(P^{\omega_m}_{\disc_m}\varphi)-\varphi\|_{L^p(\O)}
\le{}&\|\Pi_{\disc_m}(P^{\omega_m}_{\disc_m}(\varphi-\varphi_\varepsilon))\|_{L^p(\O)}\nonumber\\
&+\|\Pi_{\disc_m}(P^{\omega_m}_{\disc_m}\varphi_\varepsilon)-\varphi_\varepsilon\|_{L^p(\O)}
+\|\varphi_\varepsilon-\varphi\|_{L^p(\O)}.\label{interp.PiD}
\end{align}
By Jensen's inequality, for any $\psi\in W^{1,p}(\O)$ and $K\in\mesh_m$,
\[
|\psi^{\omega_m}_K|^p\le \frac{1}{|K|} \int_K |\omega_{m,K}(x)|^p|\psi(x)|^p\ud x\le
\frac{\|\omega_{m,K}\|_{L^\infty(K)}^p}{|K|} \int_K |\psi(x)|^p\ud x.
\]
Multiplying by $|K|$, summing over $K\in\mesh_m$, and recalling the definition of $\Pi_{\disc_m}$ gives, by choice of $\theta$,
\begin{equation}\label{key:stab}
\|\Pi_{\disc_m}(P^{\omega_m}_{\disc_m}\psi)\|_{L^p(\O)}\le \theta^d\|\psi\|_{L^p(\O)}.
\end{equation}
Using this estimate  with $\psi=\varphi-\varphi_\varepsilon$ in \eqref{interp.PiD} yields
\begin{equation}\label{cv.PDm}
\|\Pi_{\disc_m}(P^{\omega_m}_{\disc_m}\varphi)-\varphi\|_{L^p(\O)}
\le (\theta^d+1)\varepsilon+\|\Pi_{\disc_m}(P^{\omega_m}_{\disc_m}\varphi_\varepsilon)-\varphi_\varepsilon\|_{L^p(\O)}.
\end{equation}
For all $K\in\mesh_m$ and $y\in K$, by \eqref{cond.weights} and choice of $\theta$ we have
\[
|(\varphi_\varepsilon)^{\omega_m}_K-\varphi_\varepsilon(y)|=\frac{1}{|K|}\left|\int_K \omega_{m,K}(x)\left[\varphi_\varepsilon(x)-\varphi_\varepsilon(y)\right]
\ud x\right|\le \theta^d h_K \sup_{\RR^d}|\nabla\varphi_\varepsilon|.
\]
Hence, $\Pi_{\disc_m}(P^{\omega_m}_{\disc_m}\varphi_\varepsilon)\to \varphi_\varepsilon$ uniformly
on $\O$ as $m\to\infty$. Taking the superior limit as $m\to\infty$ of \eqref{cv.PDm} therefore leads to
\[
\limsup_{m\to\infty}\|\Pi_{\disc_m}(P^{\omega_m}_{\disc_m}\varphi)-\varphi\|_{L^p(\O)}\le (\theta^d+1)\varepsilon.
\]
Letting $\varepsilon\to 0$ concludes the proof that $\Pi_{\disc_m}(P^{\omega_m}_{\disc_m}\varphi)\to
\varphi$ in $L^p(\O)$ as $m\to\infty$.

The convergence of the reconstructed traces is identical, since they satisfy an equivalent
of the stability estimate \eqref{key:stab}, namely
$\|\mathbb{T}_{\disc_m}(P^{\omega_m}_{\disc_m}\psi)\|_{L^p(\partial\O)}\le \|\gamma\psi\|_{L^p(\O)}$.

\medskip

\textbf{Step 2}: convergence of the gradient reconstructions.

The proof of this convergence follows a similar reasoning, provided that
we establish the two following convergence and stability results:
\begin{equation}\label{conv.phieps}
\nabla_{\disc_m}(P^{\omega_m}_{\disc_m}\varphi_\varepsilon)\to \nabla\varphi_\varepsilon\mbox{ in $L^p(\O)^d$ as $m\to\infty$},
\end{equation}
and
\begin{equation}\label{stab.grad}
\exists C>0\,,\;\forall m\in\NN\,,\;\forall \psi\in W^{1,p}(\O)\,,\;
\|\nabla_{\disc_m}(P^{\omega_m}_{\disc_m}\psi)\|_{L^p(\O)^d}\le C\|\nabla\psi\|_{L^p(\O)^d}.
\end{equation}
In the following, $C$ denotes a generic constant that can change from one line to the other,
but depends only on $\O$, $p$ and $\theta$.

We first aim at proving \eqref{conv.phieps}.
Let $\widetilde{P}_{\polyd_m}\varphi_\varepsilon=((\varphi_\varepsilon(x_K))_{K\in\mesh},(\varphi_\varepsilon(\overline{x}_\edge))_{\edge\in\edges})\in X_{\disc_m}$. As a consequence of \cite[Lemma 12.8 and proof of Proposition 7.36]{S1}, since $\varphi_\varepsilon\in C^\infty_c(\RR^d)$,
\begin{equation}\label{cv.tildeP}
\nabla_{\disc_m}(\widetilde{P}_{\polyd_m}\varphi_\varepsilon)\to \nabla\varphi_{\varepsilon}\mbox{
in $L^p(\O)^d$ as $m\to\infty$}.
\end{equation} 
Define the following discrete $W^{1,p}$-semi-norm on $X_{\disc_m}$:
\[
\forall v\in X_{\disc_m}\,,\;\vert v\vert_{\polyd_m,p}=\left(\sum_{K\in\mesh_m}\sum_{\edge\in\edgescv}
|\edge|d_{K,\edge}\left|\frac{v_\edge-v_K}{d_{K,\edge}}\right|^p\right)^{1/p}.
\]
It follows from \cite[Lemma 5.3]{B1} that 
\begin{equation}\label{stab.grad.2}
\forall v\in X_{\disc_m}\,,\;\|\nabla_{\disc_m}v\|_{L^p(\O)^d}\le C\vert v\vert_{\polyd_m,p}.
\end{equation}
Since $\overline{x}_\edge$ is the center of mass of $\edge$, a Taylor expansion of order 2 shows that
\begin{equation}\label{est.snorm.1}
|\varphi_\varepsilon(\overline{x}_\edge)-(\varphi_\varepsilon)_\edge|
=\left|\varphi_\varepsilon(\overline{x}_\edge)-\frac{1}{|\edge|}\int_\edge \varphi_\varepsilon(x)\ud s(x)\right|\le C \|D^2\varphi_\varepsilon\|_{C_b(\RR^d)}h_\edge^2.
\end{equation}
Moreover, \cite[Lemma A.7]{DN17} yields
\begin{equation}\label{est.snorm.2}
\begin{aligned}
|\varphi_\varepsilon(x_K)-(\varphi_\varepsilon)^{\omega_m}_K|
={}&\left|\varphi_\varepsilon(x_K)-\frac{1}{|K|}\int_\edge \omega_{m,K}(x)\varphi_\varepsilon(x)\ud x\right|\\
\le{}& C
\|\varphi_\varepsilon\|_{C^2_b(\RR^d)}h_K^2.
\end{aligned}
\end{equation}
Estimates \eqref{est.snorm.1} and \eqref{est.snorm.2}, and the properties 
\[
\frac{h_\edge}{d_{K,\edge}}\le \frac{h_K}{d_{K,\edge}}\le \theta\mbox{ for all $K\in\mesh_m$ and $\edge\in\edgescv$}
\]
and (see \cite[Lemma B.2]{S1})
\[
\sum_{K\in\mesh_m}\sum_{\edge\in\edgescv}|\edge|d_{K,\edge}=
\sum_{K\in\mesh_m}\sum_{\edge\in\edgescv}d |D_{K,\edge}|=
d\sum_{K\in\mesh_m}|K|=d|\O|
\]
show that $\vert \widetilde{P}_{\polyd_m}\varphi_\varepsilon-P^{\omega_m}_{\disc_m}\varphi_\varepsilon\vert_{\polyd_m,p}\le C\|\varphi_\varepsilon\|_{C^2_b(\RR^d)}h_{\mesh_m}$.
Applying then \eqref{stab.grad.2} to $v=\widetilde{P}_{\polyd_m}\varphi_\varepsilon-P^{\omega_m}_{\disc_m}\varphi_\varepsilon$ gives
\[
\nabla_{\disc_m}(\widetilde{P}_{\polyd_m}\varphi_\varepsilon)-
\nabla_{\disc_m}(P^{\omega_m}_{\disc_m}\varphi_\varepsilon)\to 0\mbox{ in $L^p(\O)^d$ as $m\to\infty$}.
\]
Combined with \eqref{cv.tildeP}, this establishes \eqref{conv.phieps}.

Let us now turn to the stability estimate \eqref{stab.grad}. By \cite[Proposition 7.15]{S1}, 
\begin{equation}\label{est.P1}
\vert P^1_{\disc_m}\psi\vert_{\polyd_m,p}\le C\|\nabla\psi\|_{L^p(\O)^d}
\end{equation}
where $P^1_{\disc_m}$ is the interpolant \eqref{def:Ppolyd} computed with the constant
weights $\omega_K=1$. Let us estimate $\vert P^1_{\disc_m}\psi-P^{\omega_m}_{\disc_m}\psi\vert_{\polyd_m,p}$. For $K\in\mesh_m$, since $\frac{1}{|K|}\int_K \omega_{m,K}(x)\ud x=1$, we can write
\begin{align*}
\left|(P^1_{\disc_m}\psi)_K-(P^{\omega_m}_{\disc_m}\psi)_K\right|
={}&\left|\frac{1}{|K|}\int_K\psi(y)\ud y-\frac{1}{|K|}\int_K \omega_{m,K}(x)\psi(x)\ud x\right|\\
={}&\left|\frac{1}{|K|}\int_K \omega_{m,K}(x)\left(
\frac{1}{|K|}\int_K\psi(y)\ud y-\psi(x)\right)\ud x\right|\\
\le{}& \frac{\theta^d}{|K|}\int_K \left|\frac{1}{|K|}\int_K\psi(y)\ud y-\psi(x)\right|\ud x.
\end{align*}
Use then Jensen's inequality and \cite[Lemma B.7]{S1} to write
\begin{align}
\left|(P^1_{\disc_m}\psi)_K-(P^{\omega_m}_{\disc_m}\psi)_K\right|^p
\le{}& \frac{\theta^{dp}}{|K|}\int_K\left|\frac{1}{|K|}\int_K\psi(y)\ud y-\psi(x)\right|^p\ud x\nonumber\\
\le{}& \frac{Ch_K^p}{|K|}\int_K|\nabla\psi(x)|^p\ud x.\label{est.P1P}
\end{align}
Since $P^1_{\disc_m}\psi$ and $P^{\omega_m}_{\disc_m}\psi$ have the same face values, only
the difference of their cell values is involved in the computation of
$\vert P^1_{\disc_m}\psi-P^{\omega_m}_{\disc_m}\psi\vert_{\polyd_m,p}$. Hence,
dividing \eqref{est.P1P} by $d_{K,\edge}^p$ for any $\edge\in\edgescv$, using $h_K/d_{K,\edge}\le \theta$,
multiplying by $|\edge|d_{K,\edge}$, summing over $\edge\in\edgescv$, using $\sum_{\edge\in\edgescv}|\edge|d_{K,\edge}=d|K|$, and summing over $K\in\mesh_m$ leads to
\begin{align*}
\vert P^1_{\disc_m}\psi-P^{\omega_m}_{\disc_m}\psi\vert_{\polyd_m,p}
\le{}& C \|\nabla\psi\|_{L^p(\O)^d}.
\end{align*}
Combined with \eqref{est.P1}, this shows
that $\vert P^{\omega_m}_{\disc_m}\psi\vert_{\polyd_m,p} \le C\|\nabla\psi\|_{L^p(\O)^d}$.
The estimate \eqref{stab.grad} then follows from \eqref{stab.grad.2} applied
to $v=P^{\omega_m}_{\disc_m}\psi$. \end{proof}


\bibliographystyle{siam}
\bibliography{thesisref}

\end{document}

%% file: fig-geomdam.pdf_t
\begin{picture}(0,0)%
\includegraphics{fig-geomdam.pdf}%
\end{picture}%
\setlength{\unitlength}{4144sp}%
\begingroup\makeatletter\ifx\SetFigFont\undefined%
\gdef\SetFigFont#1#2#3#4#5{%
  \reset@font\fontsize{#1}{#2pt}%
  \fontfamily{#3}\fontseries{#4}\fontshape{#5}%
  \selectfont}%
\fi\endgroup%
\begin{picture}(3260,2277)(256,-1741)
\put(271,-601){\makebox(0,0)[lb]{\smash{{\SetFigFont{12}{14.4}{\rmdefault}{\mddefault}{\updefault}{\color[rgb]{0,0,0}$\Gamma_1$}%
}}}}
\put(361,-1636){\makebox(0,0)[lb]{\smash{{\SetFigFont{8}{9.6}{\rmdefault}{\mddefault}{\updefault}{\color[rgb]{0,0,0}$(0,0)$}%
}}}}
\put(406,389){\makebox(0,0)[lb]{\smash{{\SetFigFont{8}{9.6}{\rmdefault}{\mddefault}{\updefault}{\color[rgb]{0,0,0}$(0,5)$}%
}}}}
\put(3501,-1562){\makebox(0,0)[lb]{\smash{{\SetFigFont{8}{9.6}{\rmdefault}{\mddefault}{\updefault}{\color[rgb]{0,0,0}$(7,0)$}%
}}}}
\put(3228,-1225){\makebox(0,0)[lb]{\smash{{\SetFigFont{12}{14.4}{\rmdefault}{\mddefault}{\updefault}{\color[rgb]{0,0,0}$\Gamma_1$}%
}}}}
\put(1936,-106){\makebox(0,0)[lb]{\smash{{\SetFigFont{12}{14.4}{\rmdefault}{\mddefault}{\updefault}{\color[rgb]{0,0,0}$\Gamma_3$}%
}}}}
\put(1801,-1681){\makebox(0,0)[lb]{\smash{{\SetFigFont{12}{14.4}{\rmdefault}{\mddefault}{\updefault}{\color[rgb]{0,0,0}$\Gamma_2$}%
}}}}
\put(2881,-961){\makebox(0,0)[lb]{\smash{{\SetFigFont{8}{9.6}{\rmdefault}{\mddefault}{\updefault}{\color[rgb]{0,0,0}$(6,1)$}%
}}}}
\put(1171,389){\makebox(0,0)[lb]{\smash{{\SetFigFont{8}{9.6}{\rmdefault}{\mddefault}{\updefault}{\color[rgb]{0,0,0}$(2,5)$}%
}}}}
\end{picture}%

%% file: fig-geomdam2.pdf_t
\begin{picture}(0,0)%
\includegraphics{fig-geomdam2.pdf}%
\end{picture}%
\setlength{\unitlength}{4144sp}%
\begingroup\makeatletter\ifx\SetFigFont\undefined%
\gdef\SetFigFont#1#2#3#4#5{%
  \reset@font\fontsize{#1}{#2pt}%
  \fontfamily{#3}\fontseries{#4}\fontshape{#5}%
  \selectfont}%
\fi\endgroup%
\begin{picture}(5196,2355)(-1697,-2191)
\put(2386,-1861){\makebox(0,0)[lb]{\smash{{\SetFigFont{12}{14.4}{\rmdefault}{\mddefault}{\updefault}{\color[rgb]{0,0,0}$x_1=59.62$}%
}}}}
\put(3106,-2131){\makebox(0,0)[lb]{\smash{{\SetFigFont{12}{14.4}{\rmdefault}{\mddefault}{\updefault}{\color[rgb]{0,0,0}$x_1=68.56$}%
}}}}
\put(-314,-1411){\makebox(0,0)[lb]{\smash{{\SetFigFont{12}{14.4}{\rmdefault}{\mddefault}{\updefault}{\color[rgb]{0,0,0}$\Gamma_2$}%
}}}}
\put(811,-2131){\makebox(0,0)[lb]{\smash{{\SetFigFont{12}{14.4}{\rmdefault}{\mddefault}{\updefault}{\color[rgb]{0,0,0}$x_1=38.75$}%
}}}}
\put(181,-1861){\makebox(0,0)[lb]{\smash{{\SetFigFont{12}{14.4}{\rmdefault}{\mddefault}{\updefault}{\color[rgb]{0,0,0}$x_1=29.81$}%
}}}}
\put(-719,-196){\makebox(0,0)[lb]{\smash{{\SetFigFont{12}{14.4}{\rmdefault}{\mddefault}{\updefault}{\color[rgb]{0,0,0}$x_2=18$}%
}}}}
\put(-719, 29){\makebox(0,0)[lb]{\smash{{\SetFigFont{12}{14.4}{\rmdefault}{\mddefault}{\updefault}{\color[rgb]{0,0,0}$x_2=20$}%
}}}}
\put(2161,-511){\makebox(0,0)[lb]{\smash{{\SetFigFont{12}{14.4}{\rmdefault}{\mddefault}{\updefault}{\color[rgb]{0,0,0}$\Gamma_3$}%
}}}}
\put(-1214,-871){\makebox(0,0)[lb]{\smash{{\SetFigFont{12}{14.4}{\rmdefault}{\mddefault}{\updefault}{\color[rgb]{0,0,0}$\Gamma_1$}%
}}}}
\end{picture}%